\documentclass[12pt,a4paper]{amsart}

\usepackage[british]{babel}
\usepackage[utf8]{inputenc}
\usepackage{lineno}
\usepackage{mathtools,amsthm,amsmath,amsfonts,amssymb,bm,xcolor,hyperref,microtype}

\usepackage{tikz}
\usetikzlibrary{cd}
\usepackage{cm-math}

\usepackage{etoolbox} 
\newcommand*\linenomathpatch[1]{%
  \cspreto{#1}{\linenomath}%
  \cspreto{#1*}{\linenomath}%
  \csappto{end#1}{\endlinenomath}%
  \csappto{end#1*}{\endlinenomath}%
}
\linenomathpatch{equation}
\linenomathpatch{gather}
\linenomathpatch{multline}
\linenomathpatch{align}
\linenomathpatch{alignat}
\linenomathpatch{flalign}

\definecolor{cmred}{RGB}{190,64,64}
\definecolor{cmblue}{RGB}{64,97,190}
\newtheorem{theorem}{Theorem}
\newtheorem{lemma}[theorem]{Lemma}
\newtheorem{proposition}[theorem]{Proposition}

\numberwithin{equation}{section}
\newtheorem{example}{Example}
\theoremstyle{definition}

\newtheorem{assumption}[theorem]{Assumption}
\newtheorem{remark}[theorem]{Remark}

\begin{document}
\title[The Gevrey class implicit mapping with application 
       to UQ]{The Gevrey class implicit mapping theorem \\
  with application to UQ of semilinear elliptic PDEs}
\author{Helmut Harbrecht}
\author{Marc Schmidlin}
\address{Helmut Harbrecht and Marc Schmidlin
\newline
Departement Mathematik und Informatik
\newline
Universit\"at Basel
\newline
Spiegelgasse 1, 4051 Basel, Schweiz.}
\email{\{helmut.harbrecht,marc.schmidlin\}@unibas.ch}
\author{Christoph Schwab}
\address{
Christoph Schwab
\newline
Seminar für Angewandte Mathematik 
\newline 
ETH Z\"urich
\newline
R\"amistrasse 101, 8092 Z\"urich, Schweiz.}
\email{christoph.schwab@sam.math.ethz.ch}
\thanks{\textbf{Funding:} The work of HH and MS was supported
  by the Swiss National Science Foundation (SNSF) through the 
  project ``Multilevel Methods and Uncertainty Quantification in 
  Cardiac Electrophysiology'' (grant \mbox{205321\_169599}).}
\date{}

\begin{abstract}
This article is concerned with a regularity analysis of 
parametric operator equations with a perspective on
uncertainty quantification. We study the regularity of 
mappings between Banach spaces near branches of 
isolated solutions that are implicitly defined by a residual 
equation. Under $s$-Gevrey assumptions on on the residual 
equation, we establish $s$-Gevrey bounds on the Fr\'echet 
derivatives of the local data-to-solution mapping. This 
abstract framework is illustrated in a proof of regularity 
bounds for a semilinear elliptic partial differential equation 
with parametric and random field input.
\end{abstract}

\maketitle



\section{Introduction}
\label{sec:Intro}
The numerical approximation of quantities of interest
such as the expectation or the variance of an output functional
that depends on the solution of a partial differential equation (PDE)
with random input parameters is a well-established field of research.
Generally, the random data is one or more coefficient appearing 
in the PDE or, for example by use of the domain mapping approach 
\cite{XT06}, the domain on which the PDE is posed. 
A common
approach is to consider the random coefficients or domain 
mapping to be given by an affine parametric expansion
depending on countably many random variables, see 
\cite{CNT16,HPS16,HS3,XT06} for example. Indeed, such a 
representation can be achieved for example via the
Karhunen-Lo\`eve expansion of a random vector field.
This then gives rise to parametric, deterministic PDEs,
where the parameters are precisely the random variables,
which implies that computing the expectation or variance 
of an output functional of the parametric, deterministic PDE
may be done by evaluating countably dimensional integrals.

To be able to truncate the dimension of the integral and to 
choose appropriate quadrature rules in order to approximately
evaluate such integrals, it is necessary to analyse the regularity 
of the integrand with respect to the parameters. Different amounts 
of smoothness then justify the use of Monte Carlo quadrature, 
quasi-Monte Carlo quadratures including higher order versions 
as well as anisotropic sparse-grid based quadrature methods, 
see e.g.\ \cite{CD15,DickEtc,HHPS18,KN16,ZS20}. 
For example, the analytic dependence on the parameters 
has been shown for the second-order diffusion equation, for 
linear elasticity and for Stokes and Navier-Stokes equations
as well as for a class of semilinear elliptic problems with 
either a random diffusion coefficients or on random domains
in works such as \cite{C2S15,CDS10,CSZ18,HS13,HKS23,HPS16}.
Moreover, the regularity of the dependence on the parameters
for elliptic partial differential eigenvalue problems (EVP) has 
been considered in e.g.\ \cite{AS12,CL23,DE22,GGKSS19,HL19}.

In essence, there are two proof strategies that are used in the 
works that consider deterministic quadrature methods: the 
\emph{holomorphy} and the \emph{real-variable inductive} 
arguments. While the holomorphy argument has been successfully 
used to show the analytic dependence on the parameters for 
linear PDEs as well as EVPs and nonlinear PDEs, the 
real-variable inductive argument only has been able to show
the analytic dependence on the parameters for linear PDEs.
Specifically, for the case of EVPs, the real-variable inductive 
argument only has been able to show suboptimal Gevrey class 
non-analytic dependence and, recently, it was shown in \cite{CL23}
that this suboptimality is an artefact of the proof strategy itself. This 
artefact indeed also implies that the real-variable inductive argument
will at best allow to prove
suboptimal Gevrey class non-analytic 
dependence when considering nonlinear PDEs. 
Indeed, in \cite{CL23} 
the authors developed their so-named \emph{alternative-to-factorial 
technique} to circumvent the deficiency of the real-variable inductive 
argument and prove the optimal analytic dependence on the 
parameters for a class of elliptic EVPs.

\subsection{Layout}
\label{sec:Struct}
The structure of this article is as follows.
We first establish an abstract framework 
that concerns the regularity of mappings between Banach 
spaces in Section~\ref{sct:abstract}.
The main result we establish there is Theorem~\ref{theorem:analyticDS}.
It provides bounds on the Fr\'echet derivatives
of a local implicit mapping defined by a residual equation,
showing that the local implicit mapping can inherit
the \(s\)-Gevrey smoothness of the residual equation.
In other words, we establish the \emph{\(s\)-Gevrey class implicit mapping theorem}
that generalises the holomorphic implicit mapping theorem
and is not only of qualitative but also of quantitative nature.
Moreover, we also provide both, the qualitative and quantitative
behaviour, of \(s\)-Gevrey class mappings between Banach spaces
under composition in Theorems~\ref{theorem:Gevreycomp}
and~\ref{theorem:Gevreyparametriccomp}.

In Section~\ref{sec:PDE}, a model semilinear elliptic PDE 
is reformulated as a residual equation that relates the data, 
which we wish to consider as random, to the solution. By means of
Theorem~\ref{theorem:analyticDS} we show that the solution 
depends analytically on the data. 
Additionally, if the input uncertainty is formulated mathematically
as a random field realization of data in an \(s\)-Gevrey smooth residual 
equation, Theorem~\ref{theorem:Gevreyparametriccomp} 
implies the \(s\)-Gevrey smooth parametric
dependence of the solution on the random input parameters.

We demonstrate then in Section~\ref{sec:SpatReg}
how the \(s\)-Gevrey smooth dependence of the solution
of the model semilinear elliptic PDE on the data
can also be shown accounting for possible higher spatial regulearity.
Specifically, by several concrete choices of our abstract function space setting 
we show that the solution in $H^2$ and Kondrat'ev type spaces
depends \(s\)-Gevrey smooth on the data. 
Further possible 
extensions and concluding remarks are stated in Section~\ref{sec:conclusio}.
\subsection{Contributions}
\label{sec:Contr}
The main contributions of this article are as follows:
\begin{enumerate}
\item 
We establish the \(s\)-Gevrey class implicit mapping theorem
with a proof that is achieved by means of a novel 
modification of the real-variable inductive argument
that is distinct from
the alternative-to-factorial approach in \cite{CL23}.
\item 
We prove a result on \(s\)-Gevrey regularity of 
composition maps which is of independent interest.
\item 
We demonstrate, for a model, semilinear elliptic PDE, that the 
presently developed,
abstract framework directly implies novel parametric regularity results,
without use of the real-variable inductive argument.
\end{enumerate}
Specifically, our aim for this article is to introduce this approach,
where one considers the regularity of the data-to-solution mapping
and the regularity of the parameters-to-data mapping separately.
We additionally note that, as the approach is based on an inductive 
argument, it can also be used when one has only finite smoothness. 
\subsection{Notation}
\label{sec:Notat}
We use standard multi-index notation:
we denote the natural numbers including $0$ by $\Nbbb$
and excluding $0$ by $\Nbbb^*$. 
Moreover, for a sequence 
of natural numbers, $\alphabfm = \groupb{\alpha_n}_{n \in \Nbbb^*} 
\in \Nbbb^{\Nbbb^*}$, 
we as usual define the support of the sequence as
\begin{equation*}
  \supp \alphabfm = \groupb{n \in \Nbbb^* \,|\, \alpha_n \neq 0} .
\end{equation*}
If $\supp \alphabfm$ is of finite cardinality,
we say that $\alphabfm$ is finitely supported.
The set of finitely supported sequences of natural numbers
is then denoted by $\Nbbb^{\Nbbb^*}_f$,
while we will refer to its elements as multi-indices.
For multi-indices
$\alphabfm = \groupb{\alpha_n}_{n \in \Nbbb^*}, \betabfm 
= \groupb{\beta_n}_{n \in \Nbbb^*} \in \Nbbb^{\Nbbb^*}_f$
and a sequence of real numbers $\gammabfm 
= \groupb{\gamma_n}_{n \in \Nbbb^*} \in \Rbbb^{\Nbbb^*}$,
we use the following common notation:
\begin{align*}
  \norms{\alphabfm} &\isdef \sum_{n \in \supp \alphabfm} \alpha_n , &
  \alphabfm! &\isdef \prod_{n \in \supp \alphabfm} \alpha_n! , \\
  \binom{\alphabfm}{\betabfm} &\isdef \prod_{n \in \supp \alphabfm 
  \cup \supp \betabfm} \binom{\alpha_n}{\beta_n} , &
  \gammabfm^\alphabfm &\isdef \prod_{n \in \supp \alphabfm} \gamma_n^{\alpha_n} .
\end{align*}

Let $\Xcal$ and $\Ycal$ be Banach spaces,
then we denote the Banach space of bounded, 
linear maps from $\Xcal$ to $\Ycal$ as $\Bcal(\Xcal; \Ycal)$
and the space of bounded, $n$-linear maps from $\Xcal$ to $\Ycal$
as $\Bcal^n(\Xcal; \Ycal)$. 
Moreover, on the product Banach 
space $\Xcal \times \Ycal$, 
we use the norm
\begin{equation*}
  \norm[\big]{(x, y)}_{\Xcal \times \Ycal}
  \isdef \max\groupb[\big]{\norm{x}_{\Xcal}, \norm{y}_{\Ycal}} .
\end{equation*}

In a Sobolev space on a domain $G$,
the bracket $\groupa{\cdot,\cdot}_G$ denotes the duality pairing 
extending the scalar product of the Hilbert space $L^2(G)$ by continuity.

The symbol $\Dif$ shall denote the Fr\'echet derivative, 
and for integer $k>0$, $\Dif^k$ the corresponding 
multilinear Fr\'echet derivative of order $k$.
Equipped with a subscript, e.g.\ $\Dif_2$, 
it shall denote a partial Fr\'echet derivative.

For a finite set $A$, $|A|$ shall denote the number of elements $a\in A$.
\section{Regularity of implicit and composite mappings}
\label{sct:abstract}
We introduce an abstract form of nonlinear, implicit operator equation.
We recap assumptions of Gevrey-regularity of the residual map $R$, 
and conditions for the validity of the implicit function theorem,
ensuring the existence of a continuous data-to-solution map for the 
residual equation. We then state and prove the main result of this article:
quantitative bounds on Fr\'echet derivatives of the data-to-solution map 
under corresponding Gevrey-regularity hypotheses of the dependence of
$R$ on the data. Additionally, we also provide quantitative bounds on the 
Fr\'echet derivatives of the composition of mappings of Gevrey-regularity.
\subsection{Residual equation}
\label{sec:ResEq}
In the rest of this section,
let \(\Dcal\), \(\Ucal\) and \(\Rcal\) be real Banach spaces,
\(D \subset \Dcal\) and \(U \subset \Dcal\) open sets and 
\(R \colon D \times U \to \Rcal\) a mapping. As in \cite{C2S15}, 
the idea then is that the \emph{residual equation}:
given data $d\in D$, find $u\in U$ such that
\begin{equation} \label{eq:OpEqAbstr}
  R\group{d, u} = 0 \quad \text{in}\ \Rcal
\end{equation}
can be considered to be the general, abstract operator
equation of interest. It could constitute, for example, a 
suitable weak form of a partial differential equation (PDE) 
or boundary integral equation (BIE). 
We are looking for
suitable solutions $u$ of \eqref{eq:OpEqAbstr} 
as functions of the input data $d\in \Dcal$. 
In the domain of ``Uncertainty Quantification'' 
one wishes to quantify \emph{uncertainty propagation},
i.e., quantitative bounds on the dependence of the solution $u$
of \eqref{eq:OpEqAbstr} on the data $d$.

Specifically, 
we consider that \(d\) encodes uncertain input
data of the PDE, which uncertainty one models as random
by assuming (or computing from assimilated data) at hand a probability
measure on $\Dcal$ (equipped with its Borel sigma-algebra) charging $D$.
In this setup, 
\(u\) signifies the unknown solution, which is a solution 
of \eqref{eq:OpEqAbstr}
for some \(d\), i.e. precisely when \(R\group{d, u} = 0\),
under conditions on the residual mapping $R$
to ensure unique solvability
of the residual equation \eqref{eq:OpEqAbstr} \emph{locally}, i.e.
in a neighbourhood of nominal data $d^\star$
and a neighbourhood of an associated nominal solution $u^\star$.

Let us assume that 
\(R \in C^n(D \times U; \Rcal)\) for some \(n \in \Nbbb^*\), 
i.e.\ that 
\(R\) is \(n\)-times continuously Fr\'echet differentiable on \(D \times U\),
that there are \(d^\star \in D\) and \(u^\star \in U\),
which fulfil the residual equation, i.e.
\begin{equation*}
  R\group{d^\star, u^\star} = 0 \text{,}
\end{equation*}
and that 
\(\Dif_2 R\group{d^\star, u^\star} \in \Bcal\group{\Ucal; \Rcal}\)
exists and
is a Banach space isomorphism.
Then,
by the implicit mapping theorem for Banach spaces,
see e.g.\ \cite[Chapter XIV, Theorem 2.1]{Lang1993},
there is an open neighbourhood of \(d^\star\), \(D^\star \subset D\),
and a unique, continuous mapping 
\(S^\star \colon D^\star \to U\)
with
\(S^\star\group{d^\star} = u^\star\) 
such that
\begin{equation*}
  R\group[\big]{d, S^\star\group{d}} = 0 \text{ for all } d \in D^\star \text{.}
\end{equation*}
Furthermore,
this local solution mapping \(S^\star\) 
also inherits regularity from the residual mapping,
that is \(S^\star \in C^n(D^\star; \Ucal)\).

Naturally,
if \(R \in C^\infty(D \times U; \Rcal)\),
then this implies that the local solution mapping also fulfils
\(S^\star \in C^\infty(D^\star; \Ucal)\).
If in addition 
\(R\) is real analytic at \(\group{d^\star, u^\star}\),
then by considering the holomorphic extension of \(R\),
see \cite[p.\ 75]{AO53},
which is defined on the complex couple spaces,
see \cite[p.\ 312]{Taylor38},
enables us to use the holomorphic implicit mapping theorem,
see \cite[Theorem 10.2.1]{Dieudonne1969},
to prove that \(S^\star\) itself locally at \(d^\star\)
has a holomorphic extension and thus is real analytic there.
Real analyticity resp.\ parametric holomorphy of parametric 
solution maps is well-known to enable deterministic numerical
interpolation and quadrature approximations which converge
at rates which depend only on a suitable sparsity of the data.
In particular, convergence rates are free from the so-called
``curse of dimensionality'', see e.g.\ \cite{DGLS17,DGS16,HHPS18,ZS20}.

\subsection{Fr\'echet derivatives of the residual equation}
A drawback of the preceding argument
is that it does not provide any additional control
regarding the analyticity.
However,
more control is attainable by directly considering
the Fr\'echet derivatives of the local solution mapping.
We start with the equation
\begin{equation*}
  R\group[\big]{d, S^\star\group{d}} = 0
\end{equation*}
and take the \(n\)th Fr\'echet derivative of both sides with respect to \(d\)
in the directions \(h_1, \ldots, h_n\).
For the left-hand side,
we use the Fa\`a di Bruno formula, see \cite{CS96},
and arrive at the equation
\begin{align*}
  \sum_{\sigma \in \Pi_n} &\sum_{r=1}^{n} \sum_{i \in C(n, r)}
  \frac{1}{r!} \Dif^r R\group[\big]{d, S^\star\group{d}} \\
  \groups[\bigg]{
    &\frac{1}{i_1!}
    \group[\big]{
      \Dif^{i_1} \Id\group{d}\groups{h_{\sigma\group{1}}, \ldots, h_{\sigma\group{i_1}}},
      \Dif^{i_1} S^\star\group{d}\groups{h_{\sigma\group{1}}, \ldots, h_{\sigma\group{i_1}}}
    },
    \ldots , \\
    &\frac{1}{i_r!}
    \group[\big]{
      \Dif^{i_r} \Id\group{d}\groups{h_{\sigma\group{n-i_r+1}}, \ldots, h_{\sigma\group{n}}},
      \Dif^{i_r} S^\star\group{d}\groups{h_{\sigma\group{n-i_r+1}}, \ldots, h_{\sigma\group{n}}}
    }
  } = 0 \text{.}
\end{align*}
Here,
\(C(n, r)\) is the set of compositions of the natural number \(n\)
into \(r\) positive integers, given by
\begin{equation}\label{eq:defCnr}
  C(n, r) \isdef
  \groupb[\bigg]{
    \group[\big]{i_1, \ldots, i_r}
    \in \group[\big]{\Nbbb}^r :
    \sum_{k=1}^r i_k = n \text{ and } i_k \in \Nbbb^* \text{ for all } 1 \leq k \leq r
  },
\end{equation}
and \(\Pi_n\) is the set of all permutations of \(\groupb{1, \ldots, n}\).

By simplifying the terms with \(r = 1\) and reordering the equation,
we arrive at the following formula
for the \(n\)th Fr\'echet derivative of \(S^\star\),
\begin{align*}
  &\Dif^n S^\star\group{d}\groups{h_1, \ldots, h_n} \\
  &\quad= - \group[\Big]{\Dif_2 R\group[\big]{d, S^\star\group{d}}}^{-1}
  \groups[\Bigg]{
    \Dif_1 R\group[\big]{d, S^\star\group{d}} \groups[\big]{\Dif^n \Id\group{d}\groups{h_1, \ldots, h_n}} \\
    &\quad\quad\quad + \sum_{\sigma \in \Pi_n} \sum_{r=2}^{n} \sum_{i \in C(n, r)}
    \frac{1}{r!} \Dif^r R\group[\big]{d, S^\star\group{d}} \\
    &\quad\quad\quad\quad \groups[\bigg]{
      \frac{1}{i_1!}
      \group[\big]{
        \Dif^{i_1} \Id\group{d}\groups{h_{\sigma\group{1}}, \ldots, h_{\sigma\group{i_1}}},
        \Dif^{i_1} S^\star\group{d}\groups{h_{\sigma\group{1}}, \ldots, h_{\sigma\group{i_1}}}
      },
      \ldots , \\
      &\quad\quad\quad\quad\quad \frac{1}{i_r!}
      \group[\big]{
        \Dif^{i_r} \Id\group{d}\groups{h_{\sigma\group{n-i_r+1}}, \ldots, h_{\sigma\group{n}}},
        \Dif^{i_r} S^\star\group{d}\groups{h_{\sigma\group{n-i_r+1}}, \ldots, h_{\sigma\group{n}}}
      }
    }
  } \text{.}
\end{align*}
Importantly,
taking the inverse of \(\Dif_2 R\group[\big]{d, S^\star\group{d}}\)
is possible when \(d\) is close enough to \(d^\star\),
because continuity of \(\Dif_2 R\) and \(S^\star\) imply that
\(\Dif_2 R\group[\big]{d, S^\star\group{d}}\) then also
necessarily is a Banach space isomorphism. 
In particular,
for \(n = 1\) and for $h \in \Dcal$, 
\begin{equation}\label{eq:firstDS}
  \Dif S^\star\group{d}\groups{h}
  = - \group[\Big]{\Dif_2 R\group[\big]{d, S^\star\group{d}}}^{-1}
  \groups[\Big]{
    \Dif_1 R\group[\big]{d, S^\star\group{d}} \groups{h} 
  },
\end{equation}
while for \(n \geq 2\) and for $h_i \in \Dcal$, $i=1,\ldots,n$,
\begin{equation}\label{eq:higherDS}
  \begin{aligned}
  &\Dif^n S^\star\group{d}\groups{h_1, \ldots, h_n} \\
  &\quad = - \group[\Big]{\Dif_2 R\group[\big]{d, S^\star\group{d}}}^{-1}
  \groups[\Bigg]{
    \sum_{\sigma \in \Pi_n} \sum_{r=2}^{n} \sum_{i \in C(n, r)}
    \frac{1}{r!} \Dif^r R\group[\big]{d, S^\star\group{d}} \\
    &\quad\quad\quad \groups[\bigg]{
      \frac{1}{i_1!}
      \group[\big]{
        \Dif^{i_1} \Id\group{d}\groups{h_{\sigma\group{1}}, \ldots, h_{\sigma\group{i_1}}},
        \Dif^{i_1} S^\star\group{d}\groups{h_{\sigma\group{1}}, \ldots, h_{\sigma\group{i_1}}}
      },
      \ldots , \\
      &\quad\quad\quad\quad \frac{1}{i_r!}
      \group[\big]{
        \Dif^{i_r} \Id\group{d}\groups{h_{\sigma\group{n-i_r+1}}, \ldots, h_{\sigma\group{n}}},
        \Dif^{i_r} S^\star\group{d}\groups{h_{\sigma\group{n-i_r+1}}, \ldots, h_{\sigma\group{n}}}
      }
    }
  } \text{.}
  \end{aligned}
\end{equation}

\subsection{Regularity estimates}
\label{sec:RegEst}
With the formulas for the Fr\'echet derivatives of the 
local solution mapping at hand, we now consider the case,
where the residual map \(R\) is locally \(s\)-Gevrey at 
\(\group{d^\star, u^\star}\) for some \(s \in \Rbbb_{\geq1}\). 
\subsubsection{Assumptions}
\label{sec:Ass}
We list and comment on the assumptions behind our results.
\begin{assumption}\label{ass:GsReg}
\begin{enumerate}
\item
We \emph{assume}
existence of open, nonempty neighbourhoods \(\tilde{D} \subset D\) 
of \(d^\star\) and \(\tilde{U} \subset U\) of \(u^\star\) 
as well as two numbers \(\varsigma, \digamma \in \Rbbb_{\geq 0}\)
such that
\begin{equation} \label{eq:DataReg}
  \norm[\big]{\Dif^n R(d, u)}_{\Bcal^n\group{\Dcal \times \Ucal; \Rcal}}
  \leq (n!)^s \varsigma \digamma^n
\end{equation}
holds for some $s\geq 1$ and for
all \(d \in \tilde{D}\), \(u \in \tilde{U}\) and \(n \in \Nbbb\).
Possibly reducing the size of the set \(D^\star\) in \eqref{eq:OpEqImpl}
and using the continuity of \(S^\star\),
we from here on assume that we have \(D^\star \subset \tilde{D}\)
and \(S^\star\group{D^\star} \subset \tilde{U}\).
In the analytic case, i.e.\ when \(s = 1\), 
Pringsheim's Theorem, see e.g.\ \cite[p.\ 169]{Chae1985},
states that assumption \eqref{eq:DataReg}
is equivalent to \(R\) being real analytic at \(\group{d^\star, u^\star}\).
\item
We assume that the norms of the 
inverse of the Banach space isomorphisms
\(\Dif_2 R\group[\big]{d, S^\star\group{d}}\)
are bounded uniformly 
for all \(d \in D^\star\) by a uniform constant 
\(\alpha \in \Rbbb_{>0}\),
that is, 
\begin{equation}\label{eq:OpEqImpl}
  \norm[\bigg]{\group[\Big]{\Dif_2 R\group[\big]{d, S^\star\group{d}}}^{-1}}_{\Bcal\group{\Rcal; \Ucal}}
  \leq \alpha
\end{equation}
holds for all \(d \in D^\star\).
Note that this bound is a local stability estimate
for the linearisations of the residual equation.
By continuity of \(\Dif_2 R\) and possibly reducing the size of the set \(D^\star\),
such a uniform bound is achievable precisely when we have that 
\(\Dif_2 R\group{d^\star, u^\star}\) is a Banach space isomorphism.
Without loss of generality, we require
\(\alpha \geq 1\), \(\varsigma \geq 1\) and \(\digamma \geq 1\).
\end{enumerate}
\end{assumption}
%
\subsubsection{Combinatorial results}
\label{sec:CombRes}
The proofs of our main result on $s$-Gevrey regularity
of the data-to-solution map $d\mapsto S^\star\group{d}$ in \eqref{eq:OpEqImpl}
depend in an essential manner on several combinatorial facts which we 
recapitulate here for the readers' convenience.
We remind of the definition \eqref{eq:defCnr} of $C(n, r)$.
\begin{lemma}\label{lemma:simpcombineq}
  Let \(n \in \Nbbb\),
  then for any \(r \in \groupb{1, \ldots, n}\) and
  any \(i \in C(n, r)\), we have the combinatorial inequality
  \begin{equation*}
    r! \prod_{j=1}^{r} i_j! \leq n! .
  \end{equation*}
\end{lemma}
\begin{proof}
  For the proof, one simply notes that the right-hand side counts the number
  of ways to permute the list
  \begin{equation*}
    1, 2, 3, \ldots, n
  \end{equation*}
  without any restrictions, while the left-hand side counts the number
  of permutations achievable if one segments the list into \(r\)
  sublists of lengths \(i_1, \ldots, i_r\), then first permutes
  the elements in each sublist and then permutes
  the sublists themselves.
\end{proof}

In view of the type of bounds appearing, defining \(s\)-Gevrey smoothness,
we consider bounds for a particular sequence \(\group{\kappa_n}_{n \in \Nbbb^*}\).

\begin{lemma}\label{lemma:kappabound}
  The sequence \(\group{\kappa_n}_{n \in \Nbbb^*}\) recursively defined by
  \begin{equation}\label{eq:kappan}
    \kappa_n = \sum_{r=2}^{n} \sum_{i \in C(n, r)} \prod_{j=1}^r \kappa_{i_j}
  \end{equation}
  for \(n \geq 2\) and \(\kappa_1 = 1\)
  are the Schr\"oder--Hipparchus numbers,
  also sometimes called the little Schr\"oder numbers or
  the super-Catalan numbers.
  The $\kappa_n$ are bounded by
  \begin{equation*}
    \kappa_n \leq c_\kappa^{n-1}\;\quad 
    \text{for all}\ n \in \Nbbb^*,\quad \text{where}\ c_\kappa \isdef 3 + \sqrt{8}.
  \end{equation*}
  This bound is optimal in the sense that 
  if \(c_1, c_2 \in \Rbbb_{>0}\) are two constants
  for which \(\kappa_n \leq c_1 c_2^{n-1}\) for all $n\in \Nbbb$,
  then \(c_1 \geq 1\) and \(c_2 \geq c_\kappa\).
\end{lemma}

\begin{proof}
  The \(\kappa_n\) are precisely
  the Schr\"oder--Hipparchus numbers as defined in \cite{Schroeder70},
  see e.g.\ \cite{Stanley1997}.
  Furthermore, 
  in \cite[p.\ 57]{Comtet1974} it is shown that 
  the Schr\"oder--Hipparchus numbers satisfy the three-term recursion
  \begin{equation*}
    \kappa_{n+1} = \frac{6n - 3}{n + 1} \kappa_n - \frac{n - 2}{n + 1}\kappa_{n-1}, 
    \quad n\geq 2, \quad \kappa_1 = \kappa_2 = 1.
  \end{equation*}
  We now prove that
  \begin{equation*}
    \kappa_{n+1} \leq c_\kappa \kappa_n
  \end{equation*}
  holds for all \(n \geq 1\) by induction.
  For \(n = 1\), this is obviously true.
  Therefore, we assume that \(n \geq 2\).
  The induction hypothesis implies that
  $\kappa_{n-1} \geq \kappa_n/c_\kappa$ holds.
  Inserting this into the three-term recursion yields 
  \begin{equation*}
    \kappa_{n+1} \leq \bigg(\frac{6n - 3}{n + 1} - \frac{1}{c_\kappa}\frac{n - 2}{n + 1}\bigg) \kappa_n .
  \end{equation*}
  Straightforward calculation gives
  \begin{equation*}
    c_\kappa^2 - c_\kappa \bigg(\frac{6n - 3}{n + 1}\bigg) + \bigg(\frac{n - 2}{n + 1}\bigg)
    = \frac{24 + 9\sqrt{8}}{n+1}
    \geq 0.
  \end{equation*}
  Hence,
  \begin{equation*}
    \frac{6n - 3}{n + 1} - \frac{1}{c_\kappa}\frac{n - 2}{n + 1} \leq c_\kappa
  \end{equation*}
  holds, which proves \(\kappa_{n+1} \leq c_\kappa \kappa_n\)
  and concludes the induction.
  Obviously, this now implies the bound \(\kappa_n \leq c_\kappa^{n-1}\).
  
  Lastly,
  let \(c_1, c_2\in\Rbbb_{>0}\) be two constants for which 
  the bound \(\kappa_n \leq c_1 c_2^{n-1}\) holds. Then, 
  inserting \(n = 1\) directly yields \(1 = \kappa_1 \leq c_1\).
  In \cite[p.\ 539]{KnuthTAOCPI97} it is proven that the asymptotic 
  behaviour of the Schr\"oder--Hipparchus numbers is given by
  \begin{equation*}
    \kappa_n
    \sim \frac{1}{4} \sqrt{\frac{\sqrt{18}-4}{\pi}} n^{-3/2} c_\kappa^n,
  \end{equation*}
  which clearly induces that \(c_2 \geq c_\kappa\) must hold.
\end{proof}
\subsubsection{$s$-Gevrey regularity}
\label{sec:sGev}
We start by establishing a first growth bound on the 
differentials of the data-to-solution mapping $S^\star$.
\begin{lemma}\label{lemma:boundsDS}
  Under Assumption~\ref{ass:GsReg}, 
  the Fr\'echet derivatives of the local solution mapping \(S^\star\)
  are bounded as follows: 
  with the constants $\alpha, \varsigma, \digamma \geq 1$ from \eqref{eq:DataReg},
  \begin{equation}\label{eq:boundsDS}
  \forall n\in \Nbbb^*: \;\; 
  \sup_{d\in D^\star} 
    \norm[\big]{\Dif^n S^\star(d)}_{\Bcal^n\group{\Dcal; \Ucal}}
    \leq (n!)^s \alpha^{2n-1} \varsigma^{2n-1} \digamma^{3n-2} \kappa_n
    \;.
  \end{equation}
  Here, the sequence \(\group{\kappa_n}_{n \in \Nbbb^*}\) is recursively defined 
 as in \eqref{eq:kappan}.
 \end{lemma}
\begin{proof}
  We first note that obviously \(\kappa_n \geq 1\) holds for all \(n \in \Nbbb^*\).
  With this, we consider \(n = 1\).
  Then,
  taking the norm of equation~\eqref{eq:firstDS}
  and inserting the bounds yields
  \begin{align*}
    \norm[\big]{\Dif S^\star\group{d}}_{\Bcal\group{\Dcal; \Ucal}}
    &\leq \norm[\bigg]{\group[\Big]{\Dif_2 R\group[\big]{d, S^\star\group{d}}}^{-1}}_{\Bcal\group{\Rcal; \Ucal}}
    \norm[\Big]{\Dif_1 R\group[\big]{d, S^\star\group{d}}}_{\Bcal\group{\Dcal; \Rcal}} \\
    &\leq \alpha 1! \varsigma \digamma
    = (1!)^s \alpha \varsigma \digamma \kappa_1,
  \end{align*}
  proving the assertion for \(n = 1\).

  The rest of the proof is by induction:
  Let \(n \geq 2\),
  then taking the norm of equation~\eqref{eq:higherDS} leads to
  \begin{align*}
    &\norm[\big]{\Dif^n S^\star\group{d}}_{\Bcal^n\group{\Dcal; \Ucal}} 
    \\
    &\quad \leq \norm[\bigg]{\group[\Big]{\Dif_2 R\group[\big]{d, S^\star\group{d}}}^{-1}}_{\Bcal\group{\Rcal; \Ucal}}
    \sum_{\sigma \in \Pi_n} \sum_{r=2}^{n} \sum_{i \in C(n, r)}
    \frac{1}{r!} \norm[\Big]{\Dif^r R\group[\big]{d, S^\star\group{d}}}_{\Bcal^r\group{\Dcal \times \Ucal; \Rcal}} 
    \\
    &\qquad\qquad\qquad\qquad\qquad\qquad \prod_{j=1}^{r} \frac{1}{i_j!}
    \max \group[\Big]{
      \norm[\big]{\Dif^{i_j} \Id\group{d}}_{\Bcal^{i_j}\group{\Dcal; \Dcal}},
      \norm[\big]{\Dif^{i_j} S^\star\group{d}}_{\Bcal^{i_j}\group{\Dcal; \Ucal}}
    } \text{.}
  \end{align*}
  Next,
  by inserting all the bounds and noting that
  we have
  \begin{equation*}
    \norm[\big]{\Dif^{k} \Id\group{d}}_{\Bcal^{k}\group{\Dcal; \Dcal}}
    \leq 1
    \leq (k!)^s \alpha^{2k-1} \varsigma^{2k-1} \digamma^{3k-2} \kappa_{k}
  \end{equation*}
  for all \(k \in \Nbbb^*\),
  we can calculate
  \begin{align*}
    &\norm[\big]{\Dif^n S^\star\group{d}}_{\Bcal^n\group{\Dcal; \Ucal}} \\
    &\quad \leq \alpha
    \sum_{\sigma \in \Pi_n} \sum_{r=2}^{n} \sum_{i \in C(n, r)}
    \frac{1}{r!} (r!)^s \varsigma \digamma^r \prod_{j=1}^{r} \frac{1}{i_j!}
    (i_j!)^s \alpha^{2i_j-1} \varsigma^{2i_j-1} \digamma^{3i_j-2} \kappa_{i_j} \\
    &\quad =
    n! \sum_{r=2}^{n} \sum_{i \in C(n, r)} (r!)^{s-1} \alpha \varsigma \digamma^r
    \prod_{j=1}^{r} (i_j!)^{s-1} \alpha^{2i_j-1} \varsigma^{2i_j-1} \digamma^{3i_j-2} \kappa_{i_j} \\
    &\quad =
    n! \sum_{r=2}^{n} \alpha^{2n-r+1} \varsigma^{2n-r+1} \digamma^{3n-r}
    \sum_{i \in C(n, r)} (r!)^{s-1} \prod_{j=1}^{r} (i_j!)^{s-1} \kappa_{i_j} \text{.}
  \end{align*}
  For any \(r \in \groupb{1, \ldots, n}\) and
  any \(i \in C(n, r)\), we have
  \begin{equation*}
    (r!)^{s-1} \prod_{j=1}^{r} (i_j!)^{s-1} \leq (n!)^{s-1}
  \end{equation*}
  by Lemma~\ref{lemma:simpcombineq}. Hence, by 
  inserting this, we arrive at the asserted bound,
  \begin{align*}
    \norm[\big]{\Dif^n S^\star\group{d}}_{\Bcal^n\group{\Dcal; \Ucal}}
    &\leq (n!)^s \alpha^{2n-1} \varsigma^{2n-1} \digamma^{3n-2}
    \sum_{r=2}^{n} \sum_{i \in C(n, r)} \prod_{j=1}^{r} \kappa_{i_j} \\
    &= (n!)^s \alpha^{2n-1} \varsigma^{2n-1} \digamma^{3n-2} \kappa_n \text{.} \qedhere
  \end{align*}
\end{proof}

\begin{remark}\label{remark:locality}
  We note that the proof for the bounds
  \begin{equation*}
    \norm[\big]{\Dif^n S^\star(d)}_{\Bcal^n\group{\Dcal; \Ucal}}
    \leq (n!)^s \alpha^{2n-1} \varsigma^{2n-1} \digamma^{3n-2} \kappa_n
  \end{equation*}
  only requires the bounds
  \begin{equation*}
    \norm[\big]{\Dif^n R\group[\big]{d, S^\star(d)}}_{\Bcal^n\group{\Dcal \times \Ucal; \Rcal}}
    \leq (n!)^s \varsigma \digamma^n
  \end{equation*}
  and \eqref{eq:OpEqImpl}
  to hold at the specific \(d \in D^\star\) one is considering.
\end{remark}

We note that the preceding proof essentially has the form
of what is called the real-variable inductive argument in \cite{CL23}.
However, since we are considering Fr\'echet derivatives,
we are able to forgo the use of multi-indices. Moreover,
the bounds in equation \eqref{eq:boundsDS} can be rewritten as
\begin{equation*}
  \norm[\big]{\Dif^n S^\star(d)}_{\Bcal^n\group{\Dcal; \Ucal}}
  \leq (n!)^s \frac{1}{\alpha \varsigma \digamma^2} \group[\big]{\alpha^{2} \varsigma^{2} \digamma^{3}}^n \kappa_n \text{.}
\end{equation*}
Combining Lemmas~\ref{lemma:boundsDS} and~\ref{lemma:kappabound}
proves: 

\begin{theorem}\label{theorem:analyticDS}
The Fr\'echet derivatives of the local data-to-solution
map \(S^\star\) are bounded as follows: 
 \begin{equation}\label{eq:analyticDS}
 \forall n\in \Nbbb^*: \;\; 
  \sup_{d\in D^\star} 
    \norm[\big]{\Dif^n S^\star(d)}_{\Bcal^n\group{\Dcal; \Ucal}}
  \leq (n!)^s \tilde{\varsigma} \tilde{\digamma}^n .
  \end{equation}
  Herein, with $\alpha, \varsigma, \digamma \geq 1$ 
  as defined in \eqref{eq:DataReg} and \eqref{eq:OpEqImpl} 
  above,
  \begin{equation*}
    \tilde{\varsigma} := \frac{1}{c_\kappa \alpha \varsigma \digamma^2}
    \quad\text{and}\quad
    \tilde{\digamma} := c_\kappa \alpha^{2} \varsigma^{2} \digamma^{3} \text{.}
  \end{equation*}
\end{theorem}
Theorem~\ref{theorem:analyticDS} establishes the \(s\)-Gevrey
smoothness of the local data-to-solution map $S^\star:D^\star\to \Ucal$
and, thus, that an \emph{\(s\)-Gevrey class implicit mapping theorem} holds 
for all \(s \in \Rbbb_{\geq 1}\). In particular, the Gevrey regularity index $s\geq 1$ 
in the regularity assumption \eqref{eq:DataReg} on the residual equation
is inherited by the data-to-solution map.
Furthermore, Theorem \ref{theorem:analyticDS} 
reestablishes for the particular case \(s = 1\)
the real analyticity of the local solution mapping \(S^\star\) 
(without the use of holomorphy arguments as e.g.\ in \cite{C2S15}) 
using Pringsheim's Theorem,
yielding a purely real analytic proof of the real analytic implicit mapping theorem.
Indeed in this case, the bounds~\eqref{eq:analyticDS}
provide quantitative bounds on the domain of analyticity of $S^\star$.
By using the Cauchy--Hadamard formula for example, we know that the radius
of convergence for the Taylor series of \(S^\star\) at every \(d \in D^\star\)
is at least
\begin{equation}\label{eq:radius}
  \frac{1}{\tilde{\digamma}} = \frac{1}{c_\kappa \alpha^{2} \varsigma^{2} \digamma^{3}} .
\end{equation}
\subsection{Gevrey-regularity of composite and parametric mappings}
\label{sec:RegCmp}
Since in applications the quantity of interest may not always
simply be the solution itself but rather some other derived quantity,
we now additionally supply bounds for the Fr\'echet derivatives of
the composition of Gevrey-regular mappings.

\begin{theorem}\label{theorem:Gevreycomp}
  Let $\Xcal_1$, $\Xcal_2$ and $\Xcal_3$ be real Banach spaces
  and consider mappings $M_1 \colon X_1 \to \Xcal_2$ and $M_2 \colon X_2 \to \Xcal_3$,
  where $X_1 \subset \Xcal_1$ and $X_2 \subset \Xcal_2$
  are open and $M_1(X_1) \subset X_2$ holds.
  Furthermore, let $x_1 \in X_1$ and set $x_2 \isdef M_1(x_1)$.
  If $M_j$ is $s_j$-Gevrey at an $x_j \in X_j$ for some $s_j \in \Rbbb_{\geq 1}$,
  i.e.\ there is an open neighbourhood $N_j \subset X_j$ of $x_j$
  and constants $\mu_j, \nu_j \in \Rbbb_{\geq 0}$
  such that
  \begin{equation*}
    \forall n\in \Nbbb^*: \;\; 
    \sup_{x \in N_j} 
    \norm[\big]{\Dif^n M_j(x)}_{\Bcal^n\group{\Xcal_j; \Xcal_{j+1}}}
    \leq (n!)^{s_j} \mu_j \nu_j^n
  \end{equation*}
  holds, where $j \in \{1, 2\}$.
  Then, $M \isdef M_2 \circ M_1 \colon X_1 \to \Xcal_3$
  is $s$-Gevrey at an $x_1 $ for $s \isdef \max \{s_1, s_2\}$.
  Specifically, for $N \isdef N_1 \cap M_1^{-1}(N_2)$
  and
  \begin{equation*}
    \mu = \mu_2 \frac{\nu_2 \mu_1}{\nu_2 \mu_1 + 1}
    \quad\text{and}\quad
    \nu = (\nu_2 \mu_1 + 1) \nu_1,
  \end{equation*}
  we have
  \begin{equation*}
    \forall n\in \Nbbb^*: \;\; 
    \sup_{x \in N} 
    \norm[\big]{\Dif^n M(x)}_{\Bcal^n\group{\Xcal_1; \Xcal_3}}
    \leq (n!)^s \mu \nu^n \;.
  \end{equation*}
\end{theorem}
\begin{proof}
  As both, $M_1$ and $M_2$, are infinitely Fr\'echet differentiable 
  on $N_1$ and $N_2$, respectively, it is clear that $M$ also is
  infinitely Fr\'echet differentiable on $N$. Thus,
  the Fa\`a di Bruno formula gives us the
  following formula for the $n$th derivative of $M$ at an $x \in N$,
  \begin{align*}
    &\Dif^n M (x) \groups{h_1, \ldots, h_n} \\
    &\quad =\sum_{\sigma \in \Pi_n} \sum_{r=1}^{n} \sum_{i \in C(n, r)}
    \frac{1}{r!} \Dif^r M_2\group[\big]{M_1\group{x}} \\
    &\qquad\qquad \groups[\bigg]{
      \frac{1}{i_1!}
      \Dif^{i_1} M_1\group{d}\groups{h_{\sigma\group{1}}, \ldots, h_{\sigma\group{i_1}}} ,
      \ldots ,
      \frac{1}{i_r!}
      \Dif^{i_r} M_1\group{d}\groups{h_{\sigma\group{n-i_r+1}}, \ldots, h_{\sigma\group{n}}}
    } .
  \end{align*}
  Now, by taking the norm and inserting the bounds, we arrive at
  \begin{align*}
    \norm[\big]{\Dif^n M (x)}_{\Bcal^n\group{\Xcal_1; \Xcal_3}}
    &\leq n! \sum_{r=1}^{n} \sum_{i \in C(n, r)}
    (r!)^{s_2-1} \mu_2 \nu_2^r
    \prod_{j=1}^r (i_j!)^{s_1-1} \mu_1 \nu_1^{i_j} \\
    &\leq (n!)^s \mu_2 \nu_1^n \sum_{r=1}^{n} \nu_2^r \mu_1^r \sum_{i \in C(n, r)} 1 ,
  \end{align*}
  where we have also used Lemma~\ref{lemma:simpcombineq}.
  Since $\norms{C(n, r)} = \binom{n-1}{r-1}$, we finally have
  \begin{equation*}
    \norm[\big]{\Dif^n M (x)}_{\Bcal^n\group{\Xcal_1; \Xcal_3}}
    \leq (n!)^s \mu_2 \nu_1^n \sum_{r=1}^{n} \binom{n-1}{r-1} \nu_2^r \mu_1^r
    = (n!)^s \mu_2 \nu_2 \mu_1 (\nu_2 \mu_1 + 1)^{n-1} \nu_1^n ,
  \end{equation*}
  which proves the assertion.
\end{proof}

In applications, the data is often given in a parametrised fashion.
Therefore, we also provide bounds for mixed partial derivatives 
of arbitrary order for the composition of mappings,
where the outer mapping is Gevrey-regular and the inner mapping has
Gevrey-regular mixed partial derivatives of arbitrary order.
\begin{theorem}\label{theorem:Gevreyparametriccomp}
  Let $\Xcal$ and $\Ycal$ be real Banach spaces
  and consider mappings $P \colon \Pcal \to \Xcal$ and $M \colon X \to \Ycal$,
  where $\Pcal \subset \Rbbb^{\Nbbb^*}$ and $X \subset \Xcal$
  are open and $P(\Pcal) \subset X$ holds.
  Furthermore, consider a parameter $\ybfm \in \Pcal$ and set $x \isdef P(\ybfm)$.
  Assume in addition that
  $P$ has $s_P$-Gevrey 
  mixed partial derivatives of arbitrary order at $\ybfm$
  with weight $\gammabfm \in \Rbbb^{\Nbbb^*}$
  and $M$ is $s_M$-Gevrey at an $x$ for some $s_P, s_M \in \Rbbb_{\geq 1}$,
  i.e.\ there are open neighbourhoods $N_P \subset \Pcal$ of $\ybfm$
  and $N_M \subset X$ of $x$
  and constants $\mu_P, \nu_P, \mu_M, \nu_M \in \Rbbb_{\geq 0}$
  such that
  \begin{equation*}
    \forall \alphabfm \in \Nbbb^{\Nbbb^*}_f \setminus \{\zerobfm\}: \;\; 
    \sup_{\ybfm \in N}
    \norm[\big]{\pdif^{\alphabfm} P(\ybfm)}_{\Xcal}
    \leq \group[\big]{\norms{\alphabfm}!}^{s_P} \mu_P
    \nu_P^{\norms{\alphabfm}} \gammabfm^{\alphabfm}
  \end{equation*}
  and
  \begin{equation*}
    \forall n\in \Nbbb^*: \;\; 
    \sup_{x \in N_M} 
    \norm[\big]{\Dif^n M(x)}_{\Bcal^n\group{\Xcal; \Ycal}}
    \leq (n!)^{s_M} \mu_M \nu_M^n
  \end{equation*}
  hold.

  Then, $Q \isdef M \circ P \colon \Pcal \to \Ycal$
  has $s$-Gevrey mixed partial derivatives of arbitrary order at $\ybfm$
  with weight $\gammabfm$
  for $s \isdef \max \{s_P, s_M\}$.
  Specifically, for $N \isdef N_P \cap P^{-1}(N_M)$
  and
  \begin{equation*}
    \mu = \mu_M \frac{\nu_M \mu_P}{\nu_M \mu_P + 1}
    \quad\text{and}\quad
    \nu = (\nu_M \mu_P + 1) \nu_P,
  \end{equation*}
  we have
  \begin{equation*}
    \forall \alphabfm \in \Nbbb^{\Nbbb^*}_f \setminus \{\zerobfm\}: \;\; 
    \sup_{\ybfm \in N}
    \norm[\big]{\pdif^{\alphabfm} Q(\ybfm)}_{\Ycal}
    \leq \group[\big]{\norms{\alphabfm}!}^s \mu
    \nu^{\norms{\alphabfm}} \gammabfm^{\alphabfm} \;.
  \end{equation*}
\end{theorem}
\begin{proof}
  As $Q = M \circ P$ and $M$ is $s_M$-Gevrey and $P$
  has $s_P$-Gevrey mixed partial derivatives of arbitrary order,
  it follows that $Q$ also has mixed partial derivatives of arbitrary order.

  For $\alphabfm \neq \zerobfm$, according to the Fa\`a 
  di Bruno formula, we have
  \begin{equation*}
    \pdif^{\alphabfm} Q(\ybfm)
    = \alphabfm! \sum_{r=1}^{\norms{\alphabfm}} \frac{1}{r!}
    \sum_{\betabfm \in C(\alphabfm, r)}
    \Dif^r M\group[\big]{P(\ybfm)}
    \groups[\big]{\pdif^{\betabfm_1} P(\ybfm),
      \ldots, \pdif^{\betabfm_r} P(\ybfm)}
    \prod_{j=1}^r \frac{1}{\betabfm_j!} ,
  \end{equation*}
  where $C(\alphabfm, r)$ is the set of multi-index compositions
  of a multi-index $\alphabfm$
  into $r$ non-vanishing multi-indices, given by
  \begin{equation*}
    C(\alphabfm, r) \isdef
    \groupb[\bigg]{
    \group[\big]{\betabfm_1, \ldots, \betabfm_r}
    \in \group[\big]{\Nbbb_{{f}}^{\Nbbb^*}}^r :
    \sum_{j=1}^r \betabfm_j = \alphabfm \text{ and } \betabfm_{j} \neq \zerobfm \text{ for all } 1 \leq j \leq r
    } .
  \end{equation*}
  Taking the norm and inserting the bounds yields
  \begin{align*}
    \norm[\big]{\pdif^{\alphabfm} Q(\ybfm)}_{\Ycal}
    &\leq \alphabfm! \sum_{r=1}^{\norms{\alphabfm}} \frac{1}{r!}
    \sum_{\betabfm \in C(\alphabfm, r)}
    (r!)^{s_M} \mu_M \nu_M^r
    \prod_{j=1}^r \frac{\group[\big]{\norms{\betabfm_j}!}^{s_P}}{\betabfm_j!} \mu_P \nu_P^{\norms{\betabfm_j}} \gammabfm^{\betabfm_j} \\
    &= \mu_M
    \nu_P^{\norms{\alphabfm}} \gammabfm^{\alphabfm}
    \sum_{r=1}^{\norms{\alphabfm}} \nu_M^r \mu_P^r (r!)^{s-1} \alphabfm! \sum_{\betabfm \in C(\alphabfm, r)}
    \prod_{j=1}^r \frac{\group[\big]{\norms{\betabfm_j}!}^{s-1} \norms{\betabfm_j}!}{\betabfm_j!} \\
    &\leq \group[\big]{\norms{\alphabfm}!}^{s-1} \mu_M
    \nu_P^{\norms{\alphabfm}} \gammabfm^{\alphabfm}
    \sum_{r=1}^{\norms{\alphabfm}} \nu_M^r \mu_P^r \alphabfm! \sum_{\betabfm \in C(\alphabfm, r)}
    \prod_{j=1}^r \frac{\norms{\betabfm_j}!}{\betabfm_j!} ,
  \end{align*}
  where we have used the combinatorial inequality from Lemma~\ref{lemma:simpcombineq}. 
  Then, using the identity
  \begin{equation*}
    \alphabfm! \sum_{C(\alphabfm, r)} \prod_{j=1}^{r} \frac{\norms{\betabfm_j}!}{\betabfm_j!}
    = \norms{\alphabfm}! \binom{\norms{\alphabfm} - 1}{r - 1}
  \end{equation*}
  from \cite[Lemma 1]{HS22} finally gives us the asserted bound
  \begin{align*}
    \norm[\big]{\pdif^{\alphabfm} Q(\ybfm)}_{\Ycal}
    &\leq \group[\big]{\norms{\alphabfm}!}^s \mu_M
    \nu_P^{\norms{\alphabfm}} \gammabfm^{\alphabfm}
    \sum_{r=1}^{\norms{\alphabfm}} (\nu_M \mu_P)^r \binom{\norms{\alphabfm} - 1}{r - 1} \\
    &= \group[\big]{\norms{\alphabfm}!}^s \mu_M \nu_M \mu_p
    (\nu_M \mu_p + 1)^{\norms{\alphabfm} - 1}
    \nu_P^{\norms{\alphabfm}} \gammabfm^{\alphabfm} . \qedhere
  \end{align*}
\end{proof}

\section{Semilinear elliptic PDE on random domains with polynomial nonlinearities}
\label{sec:PDE}
To demonstrate the application of Theorems~\ref{theorem:analyticDS}
and~\ref{theorem:Gevreyparametriccomp}
in the context of uncertainty quantification for PDEs with random data,
we consider a semilinear, elliptic model problem with polynomial
nonlinearity. For this purpose, let $\group{\Omega, \Fcal, \Pbbb}$ 
be a probability space, on which randomness for the parametric 
problem shall be modeled.
\subsection{Problem formulation}
%
Consider the following semilinear elliptic PDE 
with random data (coefficients, source term and physical domain) 
of the following, generic form: for $\omega \in \Omega$,
\begin{equation}\label{eq:PDE}
  \left\{
    \begin{alignedat}{2}
      - \Div\group[\big]{a[\omega](\xbfm) \Grad u[\omega](\xbfm)}
      + b[\omega](\xbfm) \Nfrak\group[\big]{u[\omega](\xbfm)} &= f[\omega](\xbfm)
      &\quad &\text{for } \xbfm \in G[\omega] , \\
      u[\omega](\xbfm) &= 0
      &\quad &\text{for } \xbfm \in \Gamma_{\mathrm{D}}[\omega] , \\
      \groupa[\big]{\Grad u[\omega](\xbfm), \nbfm[\omega](\xbfm)} &= g[\omega](\xbfm)
      &\quad &\text{for } \xbfm \in \Gamma_{\mathrm{N}}[\omega] .
    \end{alignedat}
  \right.
\end{equation}
The domain $G[\omega]$ is assumed to be a bounded nonempty subset 
of $\Rbbb^m$ with dimension $m \in \Nbbb^*$.

We next introduce $q \in \Rbbb_{\geq 2}$ dependent on $m$, 
with the following restrictions
\begin{equation*}
  \begin{cases}
    q < \infty, & \text{when } m \in \{1, 2\} , \\
    q \leq \frac{2m}{m-2}, & \text{else}.
  \end{cases}
\end{equation*}
Owing to the Sobolev embedding theorem, this choice guarantees
that we have the continuous (but not necessarily compact) embedding
$H^1\group[\big]{G[\omega]} \subset L^q\group[\big]{G[\omega]}$.
With this, we now require that $\Nfrak$ in \eqref{eq:PDE} is
a polynomial of degree at most $\groupf{q-1}$ with $\Nfrak(0) = 0$, i.e. 
\begin{equation}\label{eq:Npoly}
  \Nfrak(\zeta) = \sum_{j=1}^{\groupf{q-1}} \theta_j \zeta^j
\end{equation}
holds for all $\zeta \in \Rbbb$ for some coefficients $\theta_j \in \Rbbb$.
Clearly, this implies that $\Nfrak$ fulfils the polynomial growth bound
\begin{equation}\label{eq:Ngrow}
  \forall \zeta \in \Rbbb:\quad 
  \norms[\big]{\Nfrak(\zeta)} \leq c_{\Nfrak} \group[\big]{1 + \norms{\zeta}^{q-1}}
\end{equation}
for some constant $c_{\Nfrak}$.
In addition, we assume that $\Nfrak$ is
monotone, in the sense that it fulfils
\begin{equation}\label{eq:Nmonot}
\forall \zeta,\zeta' \in \Rbbb: \quad 
\big(\Nfrak(\zeta) - \Nfrak(\zeta')\big) (\zeta-\zeta') \geq 0. 
\end{equation}
As we shall show, under 
positivity assumptions on the parametric coefficient $a[\omega](\xbfm)$
and for 
non-negative coefficient $b[\omega](\xbfm)$ in \eqref{eq:PDE},
the boundary value problem \eqref{eq:PDE} gives rise to a well-posed
monotone operator equation in (a subspace of) 
the Hilbertian Sobolev space
$H^1\group[\big]{G[\omega]}$.

The reason that we first only consider polynomial nonlinearities $\Nfrak$ here
stems from the mapping properties of the Nemyckii operator associated with $\Nfrak$.
Specifically, for $1 \leq p' \leq p \leq \infty$, 
it is known that
the Nemyckii operator $N \colon L^p \to L^{p'}$, 
defined by
\begin{equation}\label{eq:Nemyckii}
  N(u) \isdef \Nfrak \circ u ,
\end{equation}
is infinitely Fr\'echet differentiable only if
$\Nfrak$ is a polynomial with a small enough degree
or $p = \infty$ and in this case its Fr\'echet derivatives
simply are given by
\begin{equation}\label{eq:DNemyckii}
  \Dif^n N(u)[u_1, \ldots, u_n] \isdef (\Nfrak^{(n)} \circ u) \cdot u_1 \cdots u_n ,
\end{equation}
see \cite[Theorems~3.12,~3.15 and~3.16]{AZ90}.
Therefore, for $m \geq 2$, as $H^1 \not\subset L^\infty$, one can 
in general only consider the polynomial nonlinearities $\Nfrak$
that we allow above, if one is interested in the analytic or \(s\)-Gevrey
smooth dependence of the solution of \eqref{eq:PDE} on
the data\footnote{Note carefully that we show such dependence 
for data-to-solution maps between function spaces in $G$ of \emph{finite smoothness}.} 
(coefficients, source term and physical domain).

If additional Sobolev regularity of weak solutions is available, 
parametric solution regularity can hold for more general nonlinearities
as we will discuss in Section~\ref{sec:SpatReg}.

\subsection{Parametric domain}
\label{sec:ParDom}
Before providing a detailed statement, 
we clarify the notion of ``random domain'' $G[\omega]$ in \eqref{eq:PDE}.
To this end,
we adopt the random domain mapping approach as introduced in \cite{XT06}
and applied e.g.\ in \cite{C2S15,HPS16,HS3} and the references there.
To formulate it,
we assume at hand a Lipschitz domain $\hat{G} \subset \Rbbb^m$,
referred to as \emph{reference domain}\footnote{The ``reference domain'' corresponds 
to the notion of ``reference configuration'' in continuum mechanics. 
It is, in general, distinct
from the ``nominal domain'' in shape-uncertainty quantification.}, 
a disjoint decomposition of its (Lipschitz-) boundary
$\hat{\Gamma} = \partial \hat{G}$ into two measurable sets,
$\hat{\Gamma} = \hat{\Gamma}_{\mathrm{D}} \cup \hat{\Gamma}_{\mathrm{N}}$,
a parametric domain mapping $\Vbfm \colon \square \to C^1(\hat{G}; \Rbbb^m)$
with $\square \isdef \groups{-\frac{1}{2}, \frac{1}{2}}^{\Nbbb^*}$
and random parameters $\Ybfm \colon \Omega \to \square$.
We assume that $\Vbfm[\ybfm]$ is a $C^1$-diffeomorphism
and fulfils the uniformity condition
\begin{equation}\label{eq:VBiLip}
  \norm[\big]{\Vbfm[\ybfm]}_{C^1(\hat{G};\Vbfm[\ybfm](\hat{G}))} \leq c_{\Vbfm}
  \quad\text{and}\quad
  \norm[\big]{(\Vbfm[\ybfm])^{-1}}_{C^1(\Vbfm[\ybfm](\hat{G});\hat{G})} \leq c_{\Vbfm}
\end{equation}
for every $\ybfm \in \square$, where $c_{\Vbfm} \geq 1$ is a fixed constant.
We also assume 
that $\hat{\Gamma}_{\mathrm{D}}$ has non-zero surface measure, 
so that a Poincar\'{e} inequality holds on the Sobolev spaces \eqref{eq:H1D} 
on $\hat{G}$.
In an abuse of notation, we will consider any function defined over $\square$
to also be defined over $\Omega$, 
by evaluating it at $\Ybfm(\omega)$ for all $\omega \in \Omega$.

With this, we set
\begin{equation*}
  G[\ybfm] \isdef V[\ybfm](\hat{G}) ,\quad
  \Gamma_{\mathrm{D}}[\ybfm] \isdef V[\ybfm](\hat{\Gamma}_{\mathrm{D}}) \quad\text{and}\quad
  \Gamma_{\mathrm{N}}[\ybfm] \isdef V[\ybfm](\hat{\Gamma}_{\mathrm{N}})
\end{equation*}
as well as
\begin{align*}
  a[\ybfm](\xbfm) &\isdef \hat{a}\group[\Big]{\group[\big]{\Vbfm[\ybfm]}^{-1}(\xbfm)} ,\quad\hphantom{wd}\quad
  b[\ybfm](\xbfm) \isdef \hat{b}\group[\Big]{\group[\big]{\Vbfm[\ybfm]}^{-1}(\xbfm)} , \\
  f[\ybfm](\xbfm) &\isdef \hat{f}\group[\Big]{\group[\big]{\Vbfm[\ybfm]}^{-1}(\xbfm)} \quad\text{and}\quad
  g[\ybfm](\xbfm) \isdef \hat{g}\group[\Big]{\group[\big]{\Vbfm[\ybfm]}^{-1}(\xbfm)} ,
\end{align*}
where $\hat{a} \in L^\infty\group{\hat{G}}$,
$\hat{b} \in L^\infty\group{\hat{G}}$,
$\hat{f} \in H^{-1}_{\mathrm{D}}\group{\hat{G}}$
and $\hat{g} \in H^{-1/2}\group{\hat{\Gamma}_{\mathrm{N}}}$.
Note that this means that we are effectively defining the boundary decomposition
as well as the functions $a$, $b$, $f$ and $g$ in Lagrangian coordinates.
Clearly, the choices directly imply that
$a[\ybfm] \in L^\infty\group[\big]{G[\ybfm]}$,
$b[\ybfm] \in L^\infty\group[\big]{G[\ybfm]}$,
$f[\ybfm] \in H^{-1}_{\mathrm{D}}\group[\big]{G[\ybfm]}$
and $g[\ybfm] \in H^{-1/2}\group[\big]{\Gamma_{\mathrm{N}}[\ybfm]}$
for every $\ybfm \in \square$.
Here, $H^{-1}_{\mathrm{D}}\group[\big]{G[\ybfm]}$ 
denotes the dual of
\begin{equation} \label{eq:H1D}
  H^1_{\mathrm{D}}\group[\big]{G[\ybfm]}
  \isdef \groupb[\big]{v \in H^1\group[\big]{G[\ybfm]} : v(\xbfm) = 0 \text{ for all } \xbfm \in \Gamma_{\mathrm{D}}[\ybfm]} .
\end{equation}
We also assume that $\hat{a}$ and $\hat{b}$ 
fulfil the ellipticity and non-negativity condition
\begin{equation}\label{eq:abell+}
  \underline{a} \isdef \essinf_{\xbfm \in \hat{G}} \hat{a}(\xbfm) > 0
  \quad\text{and}\quad
  \essinf_{\xbfm \in \hat{G}} \hat{b}(\xbfm) \geq 0 ,
\end{equation}
respectively.
For later convenience, 
we introduce the constant $c_a \isdef \min\{1, \underline{a}\}$.

Owing to the Sobolev embbeding
$H^1\group[\big]{G[\ybfm]} \subset L^q\group[\big]{G[\ybfm]}$,
the nonlinear form
\begin{equation*}
  H^1\group[\big]{G[\ybfm]} \times H^1\group[\big]{G[\ybfm]} \to \Rbbb ,
  (w, v) \mapsto \groupa[\big]{b[\ybfm] {N(w)}, v}_{G[\ybfm]}
\end{equation*}
is well-defined as $\Nfrak$ is a polynomial of at most degree $\groupf{q-1}$
and thus its Nemyckii operator $N$ from \eqref{eq:Nemyckii}
is well-defined as $\Nfrak$ fulfils the polynomial growth bound \eqref{eq:Ngrow}.
Then, 
it is straightforward to see that the variational formulation
of \eqref{eq:PDE} for every $\ybfm \in \square$ reads:
find $u[\ybfm] \in H^1_{\mathrm{D}}\group[\big]{G[\ybfm]}$ so that
for all $v \in H^1_{\mathrm{D}}\group[\big]{G[\ybfm]}$, we have
\begin{equation}\label{eq:wPDE}
  \groupa[\big]{a[\ybfm] \Grad u[\ybfm], \Grad v}_{G[\ybfm]}
  + \groupa[\big]{b[\ybfm] N\group[\big]{u[\ybfm]}, v}_{G[\ybfm]}
  = \groupa[\big]{f[\ybfm], v}_{G[\ybfm]} 
  + \groupa[\big]{g[\ybfm], v}_{\Gamma_{\mathrm{N}}[\ybfm]} .
\end{equation}
%
\subsection{Domain mapping approach}
By utilising that $\Vbfm[\ybfm]$ is a $C^1$-diffeomorphism, 
we can pull back the spatially weak 
formulation \eqref{eq:wPDE} by considering
\begin{equation*}
  \hat{u}[\ybfm](\xbfm) \isdef u[\ybfm]\group[\big]{\Vbfm[\ybfm](\xbfm)} .
\end{equation*}
Then,
we have that 
$\hat{u}[\ybfm] \in H^1_{\mathrm{D}}\group{\hat{G}}$
for every $\ybfm \in \square$ 
and for all $v \in H^1_{\mathrm{D}}\group[\big]{G[\ybfm]}$ fulfils
\begin{multline*}
  \groupa[\big]{\tilde{\Abfm}[\ybfm] \Grad \hat{u}[\ybfm], \Grad (v \circ \Vbfm[\ybfm])}_{\hat{G}}
  + \groupa[\big]{\tilde{b}[\ybfm] N(\hat{u}[\ybfm]) , v \circ \Vbfm[\ybfm]}_{\hat{G}} \\
  = \groupa[\big]{\tilde{f}[\ybfm], v \circ \Vbfm[\ybfm]}_{\hat{G}}
  + \groupa[\big]{\tilde{g}[\ybfm], v \circ \Vbfm[\ybfm]}_{\hat{\Gamma}_{\mathrm{N}}}
\end{multline*}
with
\begin{align*}
  \tilde{\Abfm}[\ybfm](\xbfm) &= \det\group[\big]{\Jbfm[\ybfm](\xbfm)} \group[\big]{\Jbfm[\ybfm](\xbfm)}^{-1} \hat{a}(\xbfm) \group[\big]{\Jbfm[\ybfm](\xbfm)}^{-\trans} , \\
  \tilde{b}[\ybfm](\xbfm) &= \det\group[\big]{\Jbfm[\ybfm](\xbfm)} \hat{b}(\xbfm) , \\
  \tilde{f}[\ybfm](\xbfm) &= \det\group[\big]{\Jbfm[\ybfm](\xbfm)} \hat{f}(\xbfm) \\
  \text{and}\quad
  \tilde{g}[\ybfm](\xbfm) &= \norm[\Big]{\group[\big]{\Jbfm[\ybfm](\xbfm)}^{-\trans} \nbfm(\xbfm)}_2 \det\group[\big]{\Jbfm[\ybfm](\xbfm)} \hat{g}(\xbfm) ,
\end{align*}
where $\Jbfm[\ybfm](\xbfm) = \Dif_{\xbfm} \group[\big]{\Vbfm[\ybfm]}(\xbfm)$.
Note that we have made use of the fact here that 
$\Vbfm[\ybfm]:\hat{G} \to G[\ybfm]$ 
is a $C^1$-diffeomorphism for every $\ybfm \in \square$.
This implies that 
$\det\group[\big]{\Jbfm[\ybfm](\xbfm)}$ has the same sign for all $\xbfm$;
without loss of generality, we assume that it is positive, 
i.e.\ orientation preserving.

As the map $H^1_{\mathrm{D}}\group[\big]{G[\ybfm]} \to H^1_{\mathrm{D}}
\group{\hat{G}},\ v \mapsto v \circ \Vbfm[\ybfm]$
is an isomorphism for every $\ybfm \in \square$,
we can replace the $v \circ \Vbfm[\ybfm]$ terms with
$v \in H^1_{\mathrm{D}}\group[\big]{G[\ybfm]}$
for some $\hat{v}$ with $\hat{v} \in H^1_{\mathrm{D}}\group{\hat{G}}$.
We thus arrive at the spatially weak formulation for the pullback:
$\hat{u}[\ybfm] \in H^1_{\mathrm{D}}\group{\hat{G}}$ 
for every $\ybfm \in \square$ fulfils
\begin{equation}\label{eq:wPDEpullback}
  \groupa[\big]{\tilde{\Abfm}[\ybfm] \Grad \hat{u}[\ybfm], \Grad \hat{v}}_{\hat{G}}
  + \groupa[\big]{\tilde{b}[\ybfm] N(\hat{u}[\ybfm]) , \hat{v}}_{\hat{G}} \\
  = \groupa[\big]{\tilde{f}[\ybfm], \hat{v}}_{\hat{G}}
  + \groupa[\big]{\tilde{g}[\ybfm], \hat{v}}_{\hat{\Gamma}_{\mathrm{N}}}
\end{equation}
for all $\hat{v} \in H^1_{\mathrm{D}}\group{\hat{G}}$.
It is straightforward to see that we have
$\tilde{\Abfm}[\ybfm] \in L^\infty\group{\hat{G}; \Rbbb^{d \times d}_{\textrm{sym}}}$,
$\tilde{b}[\ybfm] \in L^\infty\group{\hat{G}}$,
$\tilde{f}[\ybfm] \in H^{-1}_{\mathrm{D}}\group{\hat{G}}$
and $\tilde{g}[\ybfm] \in H^{-1/2}\group{\hat{\Gamma}_{\mathrm{N}}}$
for every $\ybfm \in \square$.
Moreover, for every $\ybfm \in \square$, $\tilde{\Abfm}[\ybfm]$ and
$\tilde{b}[\ybfm]$ retain their ellipticity and non-negativity condition
\begin{equation*}
  \essinf_{\xbfm \in \hat{G}} \min_{\vbfm \in \Rbbb^m \setminus 
  \{\zerobfm\}}\frac{\vbfm^\trans \tilde{\Abfm}[\ybfm](\xbfm) \vbfm}
  {\vbfm^\trans \vbfm} \geq c_a c_{\Vbfm}^{-m-2} > 0
  \quad\text{and}\quad
  \essinf_{\xbfm \in \hat{G}} \tilde{b}[\ybfm](\xbfm) \geq 0 ,
\end{equation*}
respectively.

\subsection{Residual operator equation}
\label{sec:ResOpEq}
To cast equation~\eqref{eq:wPDEpullback} into the abstract setting of
Section~\ref{sct:abstract} on a fixed domain $\hat{G}$, 
we choose
\begin{align*}
  \Dcal &\isdef L^\infty\group{\hat{G}; \Rbbb^{m \times m}_{\textrm{sym}}}
  \times L^\infty\group{\hat{G}}
  \times H^{-1}_{\mathrm{D}}\group{\hat{G}}
  \times H^{-1/2}\group{\hat{\Gamma}_{\mathrm{N}}}, \\
  \Ucal &\isdef H^1_{\mathrm{D}}\group{\hat{G}}, \\
  \Rcal &\isdef H^{-1}_{\mathrm{D}}\group{\hat{G}} = \Ucal'
\end{align*}
and the subset of admissible data
\begin{equation*}
  D \isdef D_{\Abfm}
  \times D_b
  \times H^{-1}_{\mathrm{D}}\group{\hat{G}}
  \times H^{-1/2}\group{\hat{\Gamma}_{\mathrm{N}}},
\end{equation*}
where
\begin{equation*}
  D_{\Abfm} \isdef \groupb[\bigg]{\Abfm \in L^\infty\group{\hat{G}; \Rbbb^{m \times m}_{\textrm{sym}}} : \essinf_{\xbfm \in \hat{G}} \min_{\vbfm \in \Rbbb^m \setminus \{\zerobfm\}}\frac{\vbfm^\trans \Abfm(\xbfm) \vbfm}{\vbfm^\trans \vbfm} \geq c_a c_{\Vbfm}^{-m-2} \defis c_\Abfm}
\end{equation*}
and $D_b \isdef \groupb[\big]{b \in L^\infty\group{\hat{G}} 
: \essinf_{\xbfm \in \hat{G}} b(\xbfm) \geq 0}$. For the sake 
of legibility, we associate the data $d \in \Dcal$ 
to be given by the tuple $(\Abfm, b, f, g)$. 
We also extend this to modifications of $d$, 
for example\ $d_1 = (\Abfm_1, b_1, f_1, g_1)$.

Now, we define the residual operator $R \colon D \times \Ucal \to \Rcal$
by setting
\begin{equation}\label{eq:residual}
  \group[\big]{R(d, u)}(v)
  \isdef \groupa[\big]{\Abfm \Grad u, \Grad v}_{\hat{G}}
  + \groupa[\big]{b N(u), v}_{\hat{G}}
  - \groupa[\big]{f, v}_{\hat{G}}
  - \groupa[\big]{g, v}_{\hat{\Gamma}_{\mathrm{N}}}
\end{equation}
for all $v \in \Ucal$.
Note that this is justified since $\Ucal' = \Rcal$ and
the right-hand side of equation \eqref{eq:residual} is linear in $v$.
With this residual operator $R$,
equation~\eqref{eq:wPDEpullback} can be restated as the residual equation
\begin{equation*}
  R\group[\big]{\tilde{d}[\ybfm], \hat{u}[\ybfm]} = 0 ,
\end{equation*}
where $\tilde{d} \colon \square \to D$ is the \emph{parameters-to-data mapping}
given by
\begin{equation}\label{eq:Par2Dat}
  \tilde{d}[\ybfm] \isdef \group[\big]{\tilde{\Abfm}[\ybfm], \tilde{b}[\ybfm], \tilde{f}[\ybfm], \tilde{g}[\ybfm]} .
\end{equation}
Before we can discuss the regularity of mapping that sends the data to a solution
in Section~\ref{sec:RgDtoSmap} and
the parametric regularity of solutions with the parametric data $\tilde{d}[\ybfm]$
in Section~\ref{sec:RegParMap}, we here
consider the solvability of the residual equation
\begin{equation*}
\forall d\in D: u\in \Ucal\ \text{such that}\  R(d, u) = 0\ \text{in}\ \Rcal.
\end{equation*}
First, 
we establish strong monotonicity 
of the nonlinear operator $w\mapsto R(d,w)$.
\begin{lemma}
  Let $\Nfrak$ satisfy the polynomial growth bound \eqref{eq:Ngrow}
    and the monotonicity \eqref{eq:Nmonot}.
  Then, for every $d \in D$, the operator $\Ucal \to \Rcal,\ w \mapsto R(d, w)$
  is strongly monotone in the sense of \cite[Definition~11.1]{Deimling} 
  with constant $c_{\textrm{PF}}^{-2} c_\Abfm > 0$,
  i.e. 
  \begin{equation*}
    \group[\big]{R(d, w_1) - R(d, w_2)}(w_1 - w_2)
    \geq c_{\textrm{PF}}^{-2} c_\Abfm \norm{w_1 - w_2}_{\Ucal}^2
  \end{equation*}
  holds for all $w_1, w_2 \in \Ucal$,
  where $c_{\textrm{PF}} > 1$ is the Poincar\'e-Friedrichs constant satisfying
  \begin{equation*}
    \norm{v}_{\Ucal}^2 \leq c_{\textrm{PF}}^2 \groupa[\big]{\Grad v, \Grad v}_{\hat{G}}
  \end{equation*}
  for all $v \in \Ucal$.
\end{lemma}
\begin{proof}
  We note that, by the the stated values for $q$,
  we have the continuous Sobolev embedding $H^1(\hat{G}) \subset L^q(\hat{G})$.
  Therefore, the polynomial growth bound \eqref{eq:Ngrow}
  implies that the residual map $\Ucal \to \Rcal,\ w \mapsto R(d, w)$
  is continuous. Hence, we are left to verify monotonicity.
  Obviously, we have
  \begin{multline*}
    \group[\big]{R(d, w_1) - R(d, w_2)}(w_1 - w_2)
    \\
    = \groupa[\big]{\Abfm \Grad u_1 - \Abfm \Grad u_2, \Grad u_1 - \Grad u_2}_{\hat{G}}
    + \groupa[\big]{b N(u_1) - b N(u_2), u_1 - u_2}_{\hat{G}} .
  \end{multline*}
  The asserted strong monotonicity thus follows, as by the assumed ellipticity \eqref{eq:abell+}
  \begin{equation*}
    \groupa[\big]{\Abfm \Grad u_1 - \Abfm \Grad u_2, \Grad u_1 - \Grad u_2}_{\hat{G}}
    \geq c_\Abfm \groupa[\big]{\Grad (u_1 - u_2), \Grad (u_1 - u_2)}_{\hat{G}}
  \end{equation*}
  holds and
  \begin{equation*}
    \label{eq:EllPos}
    \groupa[\big]{b N(u_1) - b N(u_2), u_1 - u_2}_{\hat{G}} \geq 0
  \end{equation*}
  holds by \eqref{eq:Nmonot} and the (assumed) non-negativity of 
  the coefficient $b(\hat{x})$ in $\hat{G}$.
\end{proof}
This now directly yields the following result.
\begin{lemma}
  Let $\Nfrak$ satisfy the polynomial growth bound \eqref{eq:Ngrow}
  and the monotonicity \eqref{eq:Nmonot}.
  Then, for every $d \in D$, there exists a unique $u \in \Ucal$ which fulfils
  the residual equation $R(d, u) = 0$.
  In addition, we have the injectivity bound
  \begin{equation}\label{eq:Sbound}
    \norm[\big]{u}_{\Ucal}
    \leq 2 c_{\textrm{PF}}^2 c_\Abfm^{-1} \norm{d}_{\Dcal} .
  \end{equation}
\end{lemma}
\begin{proof}
  As $\Ucal$ is a real, separable Hilbert space
  and $\Rcal$ is its dual
  and, for every $d \in D$, the operator $\Ucal \to \Rcal,\ u \mapsto R(d, u)$
  is strongly monotone, the existence theorem on monotone operator equations,
  \cite[Theorem~11.2]{Deimling},
  implies the existence and uniqueness of a $u \in \Ucal$ which fulfils
  the residual equation $R(d, u) = 0$.
  For the bound, we calculate by using the strong monotonicity
  \begin{equation*}
    c_{\textrm{PF}}^{-2} c_\Abfm \norm{u}_{\Ucal}^2
    \leq \group[\big]{R(d, u) - R(d, 0)} (u - 0)
    \leq \norm[\big]{R(d, 0)}_{\Rcal} \norm{u}_{\Ucal} .
  \end{equation*}
  Then, as
  \begin{multline*}
    \norms[\Big]{\big(R(d, 0)\big)(v)}
    = \norms[\Big]{-\groupa[\big]{f, v}_{\hat{G}} - \groupa[\big]{g, v}_{\hat{\Gamma}_{\mathrm{N}}}} \\
    \leq 2 \max\groupb[\big]{\norm{f}_{H^{-1}_{\mathrm{D}}(\hat{G})}, \norm{g}_{H^{-1/2}(\hat{\Gamma}_{\mathrm{N}})}} \norm{v}_{\Ucal}
    \leq 2 \norm{d}_{\Dcal} \norm{v}_{\Ucal}
  \end{multline*}
  holds for every $v\in \Ucal$, the asserted bound follows.
\end{proof}

Therefore,
there exists a unique \emph{data-to-solution mapping}
$S \colon D \to \Ucal$ such that the equation
\begin{equation*}
  R\group[\big]{d, S(d)} = 0 \ \text{in}\ \Rcal
\end{equation*}
is fulfilled for all $d \in D$.
Moreover,
we know that $S$ maps bounded nonempty subsets $B \subset D$
to bounded nonempty subsets $S(B) \subset \Ucal$
and it is straightforward to show that it is indeed not
only continuous but even locally Lipschitz continuous.
Now, the solutions of \eqref{eq:wPDEpullback} can be stated as
$\hat{u}[\ybfm] = S\group[\big]{\tilde{d}[\ybfm]}$,
where we call 
$\hat{u} \colon \square \to H^1_{\mathrm{D}}\group{\hat{G}}$
the \emph{parameters-to-solution mapping}.

Using the first Fr\'echet derivative of $N$ from \eqref{eq:DNemyckii},
the first Fr\'echet derivative of $R$ immediately implies
in addition the following result concerning the linear maps
$\Dif_2 R (d, u) \in \Bcal(\Ucal; \Rcal)$.
\begin{proposition}\label{proposition:S}
  Let $\Nfrak$ be a polynomial nonlinearity satisfying \eqref{eq:Npoly}
  and the monotonicity \eqref{eq:Nmonot}.
  Then, for all $d \in D$, $u \in \Ucal$ and $v \in \Ucal$, 
  we have
  \begin{equation*}
    \group[\big]{\Dif_2 R(d, u)\groups{u_1}}(v)
    = \groupa[\big]{\Abfm \Grad u_1, \Grad v}_{\hat{G}}
    + \groupa[\big]{b \Dif N(u)[u_1], v}_{\hat{G}} .
  \end{equation*}
  Hence, $\Dif_2 R (d, u) \in \Bcal(\Ucal; \Rcal)$ is
  strongly monotone with constant $c_{\textrm{PF}}^{-2} c_\Abfm$
  and, therefore, 
  a Banach space isomorphism with
  \begin{equation*}
    \norm[\Big]{\group[\big]{\Dif_2 R\group{d, u}}^{-1}}_{\Bcal\group{\Rcal; \Ucal}}
  \leq c_{\textrm{PF}}^2 c_\Abfm^{-1} .
  \end{equation*}
\end{proposition}

\subsection{Regularity of the data-to-solution mapping}
\label{sec:RgDtoSmap}
We shall focus on the regularity of the data-to-solution mapping $S$.
To this end, 
we first consider the structure of the residual operator $R$.
A first observation is that the polynomial nonlinearity $\Nfrak$
implies a polynomial structure of the residual operator $R$.

\begin{lemma}\label{lemma:DR}
  Let $\Nfrak$ be a polynomial nonlinearity satisfying \eqref{eq:Npoly}.
  Then, the residual operator $R \colon D \times \Ucal \to \Rcal$ 
  is a continuous polyomial map between the Banach spaces
  $\Dcal \times \Ucal$ and $\Rcal$.
  In particular, $R$ is real analytic everywhere and
  its first Fr\'echet derivative is characterised by
  \begin{multline*}
    \group[\Big]{\Dif R(d, u)\groups[\big]{(d_1, u_1)}} (v)
    = \groupa[\big]{\Abfm_1 \Grad u + \Abfm \Grad u_1, \Grad v}_{\hat{G}} \\
    + \groupa[\big]{b_1 N(u) + b \Dif N(u)[u_1], v}_{\hat{G}}
    - \groupa[\big]{f_1, v}_{\hat{G}}
    - \groupa[\big]{g_1, v}_{\hat{\Gamma}_{\mathrm{N}}} ,
  \end{multline*}
  its second Fr\'echet derivative by
  \begin{multline*}
    \group[\Big]{\Dif^2 R(d, u)\groups[\big]{(d_1, u_1), (d_2, u_2)}}(v)
    = \groupa[\big]{\Abfm_1 \Grad u_2 + \Abfm_2 \Grad u_1, \Grad v}_{\hat{G}} \\
    + \groupa[\big]{b \Dif^2 N(u)[u_1, u_2] + b_1 \Dif N(u)[u_2] + b_2 \Dif N(u)[u_1], v}_{\hat{G}}
  \end{multline*}
  and, for $n \in \Nbbb_{\geq 3}$, its $n$th Fr\'echet derivative by
  \begin{multline*}
    \group[\Big]{\Dif^n R(d, u)\groups[\big]{(d_1, u_1), (d_2, u_2), \ldots, (d_n, u_n)}} (v) \\
    = \groupa[\bigg]{b \Dif^n N(u)[u_1, u_2, \ldots, u_n] + \sum_{\sigma \in \Pi_n} b_{\sigma(1)} N^{(n-1)}(u)[u_{\sigma(2)}, \ldots, u_{\sigma(n)}], v}_{\hat{G}} .
  \end{multline*}
  Note that we thus especially have that
  $\Dif^n R(d, u) = 0$ holds for $n \geq \groupf{q-1}+2$.
\end{lemma}
\begin{proof}
  It is straightforward to see that the first, third and fourth terms
  on the right-hand side of the definition
  \begin{equation*}
    \group[\big]{R(d, u)} (v)
    \isdef \groupa[\big]{\Abfm \Grad u, \Grad v}_{\hat{G}}
    + \groupa[\big]{b N(u), v}_{\hat{G}}
    - \groupa[\big]{f, v}_{\hat{G}}
    - \groupa[\big]{g, v}_{\hat{\Gamma}_{\mathrm{N}}}
  \end{equation*}
  are continuous polynomial maps between $\Dcal \times \Ucal$ and $\Rcal$.
  For the second term, it is clear that the map
  \begin{equation*}
    L^\infty(\hat{D}) \times L^q(\hat{D}) \to L^{q/(q-1)}(\hat{D}), 
    \quad (b,u) \mapsto b N(u)
  \end{equation*}
  is a continuous polynomial map
  between $L^\infty(\hat{G}) \times L^q(\hat{G})$ and $L^{q/(q-1)}(\hat{D})$.
  As we have $\Ucal = H^{1}_{\mathrm{D}}\group{\hat{G}} \subset L^q(\hat{G})$
  and by duality $L^{q/(q-1)}(\hat{G}) \subset H^{-1}_{\mathrm{D}}\group{\hat{G}} = \Rcal$,
  it is also a continuous polynomial between $L^\infty(\hat{G}) \times \Ucal$
  and $\Rcal$.
  Thus, $R$ is a continuous polynomial between $\Dcal$ and $\Rcal$.
  Lastly, the characterisations of the derivatives of $R$ now simply
  may be calculated algebraically and by using \eqref{eq:DNemyckii}.
\end{proof}

The characterisation of the Fr\'echet derivatives of $R$ can be used to derive
bounds of $\Dif^n R$. However, the fact that $R$ is a continuous polynomial map
trivially implies the following assertion.
\begin{proposition}\label{proposition:DR}
  Let $\Nfrak$ be a polynomial nonlinearity satisfying \eqref{eq:Npoly}.
  Then, for any bounded nonempty subsets $B \subset D$ and $V \subset \Ucal$
  there exists a constant $\varsigma \geq 1$ such that
  \begin{equation*}
    \norm[\big]{\Dif^n R(d, u)}_{\Bcal^n\group{\Dcal \times \Ucal; \Rcal}}
    \leq n! \varsigma
  \end{equation*}
  holds for all $n \in \Nbbb$ and all $d \in B$ and $u \in V$.
  Note that
  $\norm[\big]{\Dif^n R(d, u)}_{\Bcal^n\group{\Dcal \times \Ucal; \Rcal}} = 0$
  holds for $n \geq \groupf{q-1}+2$.
\end{proposition}

With these bounds at hand, 
we arrive at the bounds for the Fr\'echet derivatives
of the data-to-solution mapping $S$.
\begin{theorem}\label{theorem:main}
  Let $\Nfrak$ be a polynomial nonlinearity satisfying \eqref{eq:Npoly}
  and the monotonicity \eqref{eq:Nmonot}
  and let $\alpha \isdef c_{\textrm{PF}}^2 c_\Abfm^{-1}$.
  Then, for any bounded nonempty subset $B \subset D$ 
  there exists a constant $\varsigma \geq 1$ 
  such that
  for any $d \in B$ we have
  \begin{equation*}
    \norm[\big]{\Dif^n S(d)}_{\Bcal^n\group{\Dcal; \Ucal}}
    \leq \; n! \tilde{\varsigma} \tilde{\digamma}^n
  \end{equation*}
  for all \(n \in \Nbbb^*\), 
  where
  \begin{equation*}
    \tilde{\varsigma} = \frac{1}{c_\kappa \alpha \varsigma}
    \quad\text{and}\quad
    \tilde{\digamma} = c_\kappa \alpha^{2} \varsigma^{2} \text{.}
  \end{equation*}
\end{theorem}
\begin{proof}
  Simply apply Theorem~\ref{theorem:analyticDS} with $s=1$,
  upon noting that the Propositions~\ref{proposition:DR} and~\ref{proposition:S}
  together with equation~\eqref{eq:Sbound} and Remark~\ref{remark:locality}
  guarantee that its premises are fulfilled after setting $\digamma \isdef 1$.
\end{proof}

\subsection{Regularity of the parameters-to-data and the parameters-to-solution mapping}
\label{sec:RegParMap}
Having shown the regularity of the data-to-solution mapping $S$,
we next consider the smoothness of the parameters-to-data mapping 
$\tilde{d}$ in \eqref{eq:Par2Dat} and of the parameters-to-solution
mapping $\hat{u}$ stemming from \eqref{eq:wPDEpullback}.
To this end, we make the following assumption:
\begin{assumption}\label{assumption:Vbounds2}
  The parametric domain mapping $\Vbfm \colon \square \to C^1(\hat{G}; \Rbbb^m)$
  satisfies \eqref{eq:VBiLip} and 
  admits bounded mixed partial derivatives of arbitrary order in the sense that,
  for some constants $\mu_{\Vbfm}, \kappa_{\Vbfm} > 0$, $s \geq 1$
  and a $\gammabfm \in \ell^1(\Nbbb^*)$, these satisfy
  \begin{equation*}
    \norm[\big]{\pdif^{\alphabfm}_\ybfm \Vbfm[\ybfm]}_{C^1(\hat{G}; \Rbbb^m)}
    \leq \group[\big]{\norms{\alphabfm}!}^s \mu_{\Vbfm} \kappa_{\Vbfm}^{\norms{\alphabfm}} \gammabfm^{\alphabfm}
  \end{equation*}
  for all $\ybfm\in \square$ and all multi-indices $\alphabfm$.
\end{assumption}

Now, combining Theorem~\ref{theorem:Gevreyparametriccomp}
together with the results found in \cite{HKS23,HPS16,HS22},
cf.\ especially \cite[Lemmas~3 and~4]{HS22}, 
immediately imply smoothness of the parameters-to-data mapping $\tilde{d}$.
Moreover, as the parametric domain mapping $\Vbfm$ is a bounded 
map by considering $\alphabfm = \zerobfm$, this is also true
for the parameters-to-data mapping $\tilde{d}$.
Therefore, we also have smoothness of the parameters-to-solution
mapping $\hat{u}$ by combining Theorems~\ref{theorem:Gevreyparametriccomp}
and~\ref{theorem:main}, as $\tilde{d}(\square) \subset D$ holds.
Specifically, this gives the following result:
\begin{theorem}\label{theorem:dubounds}
  Let $\Nfrak$ be a polynomial nonlinearity satisfying \eqref{eq:Npoly}
  and the monotonicity \eqref{eq:Nmonot}
  and suppose that Assumption~\ref{assumption:Vbounds2} holds.
  Then, both the parameters-to-data mapping $\tilde{d} \colon \square \to D$
  and the parameters-to-solution mapping
  $\hat{u} \colon \square \to H^1_{\mathrm{D}}\group{\hat{G}}$
  have bounded mixed partial derivatives of arbitrary order
  and there exist constants $\mu_{\tilde{d}}, \kappa_{\tilde{d}} \geq 1$
  and $\mu_{\hat{u}}, \kappa_{\hat{u}} \geq 1$ such that
  \begin{equation*}
    \norm[\big]{\pdif^{\alphabfm}_\ybfm \tilde{d}[\ybfm]}_{\Dcal}
    \leq \group[\big]{\norms{\alphabfm}!}^s 
         \mu_{\tilde{d}} \kappa_{\tilde{d}}^{\norms{\alphabfm}} \gammabfm^{\alphabfm}
  \end{equation*}
  and
  \begin{equation*}
    \norm[\big]{\pdif^{\alphabfm} \hat{u}[\ybfm]}_{H^1_{\mathrm{D}}(\hat{G})}
    \leq \group[\big]{\norms{\alphabfm}!}^s \mu_{\hat{u}} \kappa_{\hat{u}}^{\norms{\alphabfm}} 
    \gammabfm^{\alphabfm}
  \end{equation*}
  hold for all $\ybfm\in \square$ and all multi-indices $\alphabfm$.
\end{theorem}

We note that Theorems~\ref{theorem:main} and~\ref{theorem:dubounds}
themself do not actually rely on the fact
that $d \in D$ or $\tilde{d}[\ybfm]$ stem from a domain mapping.
Specifically, the bounds of the mixed partial derivatives of arbitrary order
of $\hat{u}$ are true for any parameters-to-data mapping $\tilde{d}$
that satisfies the bounds of the mixed partial derivatives of arbitrary order
given in Theorem~\ref{theorem:dubounds}, as long as
$\tilde{d}(\square) \subset D$ holds.
Therefore, Theorem~\ref{theorem:dubounds} is also applicable when considering
the semilinear elliptic PDE on a deterministic domain but with random coefficients,
as was considered in \cite{CD15,HS13} for example. Indeed, 
Theorems~\ref{theorem:main} and~\ref{theorem:dubounds} 
relate to \cite[Section 2.3 and especially Remark 2.6]{CD15}
when $s = 1$.
Theorem~\ref{theorem:main} 
implies that it is not necessary to restrict the data $d$ to a set 
$B \subset D$ that is compact in $\Dcal$ but that its boundedness suffices
in order to provide an explicit description of where a holomorphic 
extension can be defined by simply using the Cauchy--Hadamard formula,
see equation~\eqref{eq:radius}.

\section{Higher spatial regularity and non-polynomial nonlinearities}
\label{sec:SpatReg}
Up to this point, 
we have considered the solutions in their ``energy'' variational space. 
However, it is known
that analytic or Gevrey parametric regularity with
higher order spatial regularity of the solutions is mandatory for
achieving dimension-independent convergence rates
using multilevel quadrature or collocation methods, see 
\cite{GHM20,harbrecht2012multilevel,HS22,KSS12,TJWG15} for example.
Therefore, to demonstrate how higher spatial regularity of solutions
is also handled within the framework of Section~\ref{sct:abstract},
we consider particular cases of the semilinear PDE \eqref{eq:PDE}.
Especially, by leveraging the higher spatial regularity of the solution,
we also will be able to consider non-polynomial analytic as well as
$s$-Gevrey nonlinearities $\Nfrak$ here.
 
For this, we assume that $\hat{G} \subset \Rbbb^m$
has a $C^{1,1}$-smooth boundary,
that $\Gamma_{\textrm{N}} = \emptyset$, so that $H^1_{\mathrm{D}} = H^1_0$ holds,
and that the parametric domain mapping fulfils
$\Vbfm \colon \square \to C^{1,1}(\hat{G}; \Rbbb^m)$.
We assume that $\Vbfm[\ybfm]$ is a $C^{1,1}$-isomorphism
and fulfils the uniformity condition
\begin{equation}\label{eq:VBiHoel}
  \norm[\big]{\Vbfm[\ybfm]}_{C^{1,1}(\hat{G};\Vbfm[\ybfm](\hat{G}))} \leq c_{\Vbfm}
  \quad\text{and}\quad
  \norm[\big]{(\Vbfm[\ybfm])^{-1}}_{C^{1,1}(\Vbfm[\ybfm](\hat{G});\hat{G})} \leq c_{\Vbfm}
\end{equation}
for every $\ybfm \in \square$, where $c_{\Vbfm} \geq 1$ is a fixed constant.
Moreover, we restrict ourselves to the cases where $m \in \groupb{1, 2, 3}$.

We shall also assume that $\hat{a} \in W^{1,\infty}\group{\hat{G}}$,
$\hat{b} \in L^\infty\group{\hat{G}}$,
and $\hat{f} \in L^2\group{\hat{G}}$,
which means that
$a[\ybfm] \in W^{1,\infty}\group[\big]{G[\ybfm]}$,
$b[\ybfm] \in L^\infty\group[\big]{G[\ybfm]}$
and $f[\ybfm] \in L^2\group[\big]{G[\ybfm]}$
for every $\ybfm \in \square$. It is straightforward to 
see that these assumptions imply $\tilde{\Abfm}[\ybfm] \in 
W^{1,\infty}\group[\big]{\hat{G}; \Rbbb^{d \times d}_{\textrm{sym}}}$,
$\tilde{b}[\ybfm] \in L^\infty\group{\hat{G}}$ and 
$\tilde{f}[\ybfm] \in L^2_{\mathrm{D}}\group{\hat{G}}$
for every $\ybfm \in \square$.

Lastly, we assume 
that $\Nfrak \colon \Rbbb \to \Rbbb$ is monotonically increasing,
i.e.~that \eqref{eq:Nmonot} holds, and fulfils $\Nfrak(0) = 0$.
However, instead of requiring that it fulfils \eqref{eq:Npoly},
we require the stronger polynomial growth bound
that $\Nfrak$ fulfils
\begin{equation}\label{eq:Ngrow2}
\forall \zeta \in \Rbbb:\quad 
  \norms[\big]{\Nfrak(\zeta)} \leq c_{\Nfrak} \group[\big]{1 + \norms{\zeta}^{q/2}}
\end{equation}
for finite constants $c_{\Nfrak}>0$ and $q>0$
and that $\Nfrak$ is $s$-Gevrey for some fixed $s \geq 1$, i.e.\
for every compact $K \subset \Rbbb$,
there exists a constant $c_{\Nfrak,K} > 0$
such that for all $n\in \Nbbb$
\begin{equation}\label{eq:sGN}
  \sup_{\zeta \in K} | \Nfrak^{(n)}(\zeta) | \leq (c_{\Nfrak,K})^{n+1} (n!)^s \;.
\end{equation}
holds.

We observe that \eqref{eq:Ngrow2} is a global condition on $\Nfrak$,
whereas \eqref{eq:sGN} are localised to compacta $K\subset \Rbbb$.
As \eqref{eq:Ngrow}, 
condition \eqref{eq:Ngrow2} ensures global existence of weak solutions.
It also ensures their $H^2$ regularity, uniformly with respect to the data:
\eqref{eq:Ngrow2} ensures that one has 
the continuous Nemyckii operator associated with $\Nfrak$,
\begin{equation*}
  L^q\group[\big]{G[\ybfm]} \to L^2\group[\big]{G[\ybfm]},\ u \mapsto \Nfrak \circ u,
\end{equation*}
and that there holds the continuous embedding
$H^1_0\group[\big]{G[\ybfm]} \subset L^{q}\group[\big]{G[\ybfm]}$,
by the Sobolev embedding theorem,
and uniformly for every $\ybfm \in \square$.

On the other hand, condition \eqref{eq:sGN} will be sufficient to ensure
that we have an $s$-Gevrey smooth Nemyckii operator
associated with $\Nfrak$ defined as
\begin{equation}\label{eq:NemycOp}
  N \colon L^\infty\group[\big]{G[\ybfm]} \to L^\infty\group[\big]{G[\ybfm]}:  u \mapsto \Nfrak \circ u ,
\end{equation}
which we will use in combination with the continuous embedding
$H^2\group[\big]{G[\ybfm]} \subset L^\infty\group[\big]{G[\ybfm]}$,
being valid uniformly with respect to $\ybfm \in \square$,
that we have by the Sobolev embedding theorem.

\begin{example}\label{ex:Nonlin}
  We provide examples for the nonlinear term $\Nfrak$.
  \begin{enumerate}
  \item 
    A first valid example for a nonlinearity $\Nfrak$ is the cubic
    nonlinearity $\Nfrak(\zeta) = \zeta^3$. 
    Evidently, \eqref{eq:sGN} is valid with $s=1$.
    Also, \eqref{eq:Ngrow2} holds with $q=6$ and 
    $\Nfrak'(\zeta) = 3\zeta^2 \geq 0$ from which \eqref{eq:Nmonot} follows.
  \item
    A second example for $\Nfrak$ is 
    \begin{equation*}
      \Nfrak(\zeta) = \frac{\zeta^3}{1+\exp(-1/\zeta^2)} \;,\quad \zeta \ne 0.
    \end{equation*}
    For $\zeta \to 0$, the definition of $\Nfrak(\zeta)$ is completed with the 
    corresponding limits such as $\zeta^k \exp(-1/\zeta^2) \to 0$ for any
    finite $k$.
    One verifies that $\zeta\mapsto \Nfrak(\zeta)$ is smooth, 
    but not analytic, 
    and that for $0\ne \zeta \in \Rbbb$
    \begin{equation*}
      \Nfrak'(\zeta) 
      = \frac{3\zeta^2}{(1+\exp(-1/\zeta^2))}
      \left[ 1-\frac{1}{3} \zeta^{-2} \exp(-\zeta^{-2}) (1+\exp(-1/\zeta^2))^{-1} \right]
      \;,
    \end{equation*}
    so that $\Nfrak'(\zeta) > 0$ for all $\zeta\in \Rbbb$ whence
    the monotonicity of $\Nfrak$ in \eqref{eq:Nmonot} follows. 
    Evidently, 
    then also the growth condition \eqref{eq:Ngrow2} holds with $q=6$.
    Furthermore, $\Nfrak$ is $s$-Gevrey regular with $s \geq 3/2$.
  \item 
    An example of an analytic, nonpolynomial nonlinearity is 
    \begin{equation*}
      \Nfrak(\zeta) = 2 + \tanh(\zeta) 
      = 2 + \frac{\exp(\zeta)-\exp(-\zeta)}{\exp(\zeta)+\exp(-\zeta)}\;.
    \end{equation*}
    Then $\Nfrak(\zeta) \in [1,3]$ and $\Nfrak$ is analytic (i.e., $1$-Gevrey) 
    at all $\zeta\in \Rbbb$.
    Due to 
    \begin{equation*}
      \Nfrak'(\zeta) = 1/\cosh^2(\zeta) = 4/(\exp(\zeta)+\exp(-\zeta))^2 > 0
    \end{equation*}
    for all $\zeta\in \Rbbb$,
    it satisfies \eqref{eq:Nmonot} and also \eqref{eq:Ngrow2} with $q=0$.
  \item
    A nonlinearity which is \emph{not} covered is $\Nfrak(\zeta) = \exp(\zeta)$
    which appears in mathematical models of combustion, for example. 
    While $\Nfrak$ is analytic (condition \eqref{eq:sGN} even holds with $s=0$) 
    and monotone,  the polynomial growth condition \eqref{eq:Ngrow2} cannot be satisfied.
  \end{enumerate}
\end{example}

\subsection{Residual Equation}
\label{sec:ResEqnH2}
In view of equation~\eqref{eq:wPDEpullback} and 
Section~\ref{sct:abstract}, we choose
\begin{align}\label{eq:RegSpc} 
\nonumber
  \Dcal &\isdef W^{1,\infty}\group[\big]{\hat{G}; \Rbbb^{d \times d}_{\textrm{sym}}}
  \times L^\infty\group{\hat{G}}
  \times L^2\group{\hat{G}} , \\
  \Ucal &\isdef H^1_0\group{\hat{G}} \cap H^2\group{\hat{G}} , \\
\nonumber
  \Rcal &\isdef L^2\group{\hat{G}}
\end{align}
and the subset of admissible data
\begin{equation*}
  D \isdef D_{\Abfm}
  \times D_b
  \times L^2\group{\hat{G}} ,
\end{equation*}
where
\begin{equation*}
  D_{\Abfm} \isdef \groupb[\bigg]{\Abfm \in W^{1,\infty}\group[\big]{\hat{G}; \Rbbb^{m \times m}_{\textrm{sym}}} : \essinf_{\xbfm \in \hat{G}} \min_{\vbfm \in \Rbbb^m \setminus \{\zerobfm\}}\frac{\vbfm^\trans \Abfm(\xbfm) \vbfm}{\vbfm^\trans \vbfm} \geq c_a c_{\Vbfm}^{-m-2} \defis c_\Abfm}
\end{equation*}
and $D_b \isdef \groupb[\big]{b \in L^\infty\group[\big]{\hat{G}} : \essinf_{\xbfm \in \hat{G}} b(\xbfm) \geq 0}$.
For the sake of legibility, we will associate the data variable $d \in \Dcal$ 
to be given by $(\Abfm, b, f)$ and also extend this to 
modifications of $d$, i.e.\ $d_1 = (\Abfm_1, b_1, f_1)$.

Now, we define the residual operator $R \colon D \times \Ucal \to \Rcal$
by setting
\begin{equation}\label{eq:residualH2}
  R(d, u)
  \isdef - \Div\group{\Abfm \Grad u} + b N(u) - f .
\end{equation}
With this residual operator $R$,
equation~\eqref{eq:wPDEpullback} can be restated as the residual equation
\begin{equation*}
  R\group[\big]{\tilde{d}[\ybfm], \hat{u}[\ybfm]} = 0 \ \text{in}\ \Rcal,
\end{equation*}
where $\tilde{d} \colon \square \to D$ is the \emph{paramaters-to-data mapping}
given by
\begin{equation*}
  \tilde{d}[\ybfm] \isdef \group[\big]{\tilde{\Abfm}[\ybfm], \tilde{b}[\ybfm], \tilde{f}[\ybfm]} .
\end{equation*}
In this formulation, we now first consider the solvability of the residual equation
\begin{equation*}
  R(d, u) = 0
\end{equation*}
for a given $d \in D$ and unknown $u \in \Ucal$.
\begin{theorem}\label{theorem:solvH2}
  Let $\Nfrak$ satisfy the polynomial growth bound \eqref{eq:Ngrow2}
  and the monotonicity \eqref{eq:Nmonot}.
  Then, for every $d \in D$, there exists a unique $u_d \in \Ucal$ that fulfils
  the residual equation $R(d, u) = 0$ in $\Rcal$.
  Moreover, for any bounded nonempty subset $B \subset D$, there
  exists a constant $c_B > 0$ such that for all $d\in B$ holds
  \begin{equation}\label{eq:SboundH2}
    \norm{u_d}_{\Ucal}
    \leq c_B \norm{d}_{\Dcal}.
  \end{equation}
\end{theorem}
\begin{proof}
  We let $\Xcal = H^1_0\group{\hat{G}}$, then we define the operator
  $T_d \colon \Xcal \to \Xcal'$ as
  \begin{equation*}
    \group[\big]{T_d(u)}(v)
    \isdef \groupa[\big]{\Abfm \Grad u, \Grad v}_{\hat{G}}
    + \groupa[\big]{b N(u), v}_{\hat{G}}
    - \groupa[\big]{f, v}_{\hat{G}} .
  \end{equation*}
  For $u_1, u_2 \in \Ucal$, we have
  \begin{multline*}
    \group[\big]{T_d(u_1) - T_d(u_2)}(u_1 - u_2) \\
    = \groupa[\big]{\Abfm \Grad u_1 - \Abfm \Grad u_2, \Grad u_1 - \Grad u_2}_{\hat{G}}
    + \groupa[\big]{b N(u_1) - b N(u_2), u_1 - u_2}_{\hat{G}} .
  \end{multline*}
  Since
  \begin{equation*}
    \groupa[\big]{\Abfm \Grad u_1 - \Abfm \Grad u_2, \Grad u_1 - \Grad u_2}_{\hat{G}}
    \geq c_\Abfm \groupa[\big]{\Grad (u_1 - u_2), \Grad (u_1 - u_2)}_{\hat{G}}
  \end{equation*}
  holds by ellipticity and
  \begin{equation*}
    \groupa[\big]{b N(u_1) - b N(u_2), u_1 - u_2}_{\hat{G}} \geq 0
  \end{equation*}
  holds by monotonicity of $N_K$ and non-negativity of $b$,
  we have that $T_d$ is strongly monotone with
  \begin{equation*}
    \group[\big]{T_d(u_1) - T_d(u_2)}(u_1 - u_2)
    \geq c_{\textrm{PF}}^{-2} c_\Abfm \norm{u_1 - u_2}_{\Xcal}^2 .
  \end{equation*}
  Here, $c_{\textrm{PF}} > 1$ is the Poincar\'e-Friedrichs constant satisfying
  \begin{equation*}
    \norm{v}_{\Xcal}^2 \leq c_{\textrm{PF}}^2 \groupa[\big]{\Grad v, \Grad v}_{\hat{G}}
  \end{equation*}
  for all $v \in \Xcal$.
  Hence,
  as $\Xcal$ is a real, separable Hilbert space for every $d \in D$,
  the theorem on monotone operator equations,
  see \cite[Theorem~11.2]{Deimling},
  implies the existence and uniqueness of a $u_d$ that
  fulfils the operator equation $T_d(u_d) = 0$.
  Moreover, using the strong monotonicity, we have
  \begin{equation*}
    c_{\textrm{PF}}^{-2} c_\Abfm \norm{u_d - 0}_{\Xcal}^2
    \leq \group[\big]{T_d(u_d) - T_d(0)}(u_d - 0)
    \leq \groupa[\big]{- T_d(0)}(u_d) ,
  \end{equation*}
  which yields that
  \begin{equation*}
    \norm{u_d}_{\Xcal}
    \leq c_{\textrm{PF}}^2 c_\Abfm^{-1} \norm{f}_{L^2\group{\hat{G}}}
    \leq c_{\textrm{PF}}^2 c_\Abfm^{-1} \norm{d}_{\Dcal} .
  \end{equation*}

  Next, we set $w_d \isdef u_d$ and note that $w_d$ fulfils
  the equation
  \begin{equation*}
    \groupa[\big]{\Abfm \Grad w_d, \Grad v}_{\hat{G}}
    = \groupa[\big]{f - b N(u_d), v}_{\hat{G}} .
  \end{equation*}
  As $u_d$ is in $L^q\group{\hat{G}}$, we have that the term
  $f - b N(u_d)$ is an element of $L^2\group{\hat{G}}$ with
  \begin{align*}
    \norm[\big]{f - b N(u_d)}_{L^2\group{\hat{G}}}
    &\leq \norm{f}_{L^2\group{\hat{G}}} + c_{\Nfrak} \norm{b}_{L^\infty\group{\hat{G}}}
    \group[\big]{1 + \norm{u_d}_{L^q\group{\hat{G}}}^{q/2}} \\
    &\leq \norm{d}_{\Dcal} \group[\big]{1 + c_{\Nfrak} + c_{\Nfrak} c_{\textrm{PF}}^q c_\Abfm^{-q/2} \norm{d}_{\Dcal}^{q/2}} .
  \end{align*}
  Therefore, by elliptic regularity, see \cite[Theorem 8.12]{GilbargTrudinger},
  we know that $w_d \in H^2\group{\hat{G}}$ with
  \begin{equation}\label{eq:regb}
    \norm{w_d}_{H^2\group{\hat{G}}}
    \leq c_{\textrm{er}} \norm{d}_{\Dcal} \group[\big]{1 + c_{\Nfrak} + c_{\Nfrak} c_{\textrm{PF}}^q c_\Abfm^{-q/2} \norm{d}_{\Dcal}^{q/2}},
  \end{equation}
  where $c_{\textrm{er}}$ only depends on $m$, $\hat{G}$, $c_\Abfm$
  and an upper bound for $\norm{\Abfm}_{W^{1,\infty}(\hat{G})}$.

  Obviously, we thus have that $u_d \in H^1_0\group{\hat{G}} \cap H^2\group{\hat{G}}$
  and that $u_d$ indeed fulfils the residual equation
  \begin{equation*}
    R(d, u_d) = 0 .
  \end{equation*}
  Moreover, for any bounded nonempty subset $B \subset D$, 
  the assertion follows 
  by setting
  \begin{equation*}
    c_B \geq c_{\textrm{er}} \group[\big]{1 + c_{\Nfrak} + c_{\Nfrak} c_{\textrm{PF}}^q c_\Abfm^{-q/2} K^{q/2}}
  \end{equation*}
  with $K = \sup_{d \in B} \norm{d}_{\Dcal}$
  and where $c_{\textrm{er}}$ is chosen with the upper bound $K$
  for $\norm{\Abfm}_{W^{1,\infty}(\hat{G})}$.
\end{proof}

We thus know that there exists a unique, global data-to-solution 
mapping $S \colon D \to \Ucal$ such that the equation
\begin{equation*}
  R\group[\big]{d, S(d)} = 0
\end{equation*}
is fulfilled for all $d \in D$.
Moreover,
we know that $S$ maps bounded nonempty subsets $B \subset D$
to bounded nonempty subsets $S(B) \subset \Ucal$.
Hence, the solutions of \eqref{eq:wPDEpullback} here can be stated as
$\hat{u}[\ybfm] = S\group[\big]{\tilde{d}[\ybfm]}$,
where $\hat{u} \colon \square \to H^1_0\group{\hat{G}} \cap H^2\group{\hat{G}}$
is the parameters-to-solution mapping.

\subsection{Regularity of the data-to-solution mapping}
\label{sec:RegDatSolH2}
We now focus on the regularity of the data-to-solution mapping $S$.
To this end, we first consider the regularity of the Nemyckii operator
$N \colon L^\infty\group{\hat{G}} \to L^\infty\group{\hat{G}}$.
Indeed, between these spaces it turns out that $N$ inherits the differentiability
and smoothness of $\Nfrak$.
\begin{lemma}\label{lemma:DN}
  Assume that $\hat{G}$ is a bounded Lipschitz domain and 
  the $s$-Gevrey regularity \eqref{eq:sGN} of the nonlinearity $\Nfrak$.
  Then, the Nemyckii operator
  \begin{equation*}
    N \colon L^\infty\group{\hat{G}} \to L^\infty\group{\hat{G}}, u \mapsto \Nfrak \circ u
  \end{equation*}
  is $s$-Gevrey and its Fr\'echet derivatives are given by
  \begin{equation*}
    \Dif^n N(u) [u_1, \ldots, u_n](\xbfm) = \Nfrak^{(n)}\group[\big]{u(\xbfm)} \cdot u_1(\xbfm) \cdots u_n(\xbfm) .
  \end{equation*}
  Moreover, for any bounded nonempty subset $V \subset L^\infty\group{\hat{G}}$,
  there exist constants $\varsigma, \digamma \geq 1$ such that
  \begin{equation*}
    \norm[\big]{\Dif^n N(u)}_{\Bcal^n\group{L^\infty\group{\hat{G}}; L^\infty\group{\hat{G}}}}
    \leq (n!)^s \varsigma \digamma^n
  \end{equation*}
  holds for all $n \in \Nbbb$ and all $u \in V$.
\end{lemma}
\begin{proof}
  We define the Nemyckii operators
  $N_n \colon L^\infty\group{\hat{G}} \to L^\infty\group{\hat{G}},\ u \mapsto \Nfrak^{(n)} \circ u$
  for all $n \in \Nbbb$.
  For any arbitrary open bounded nonempty subset $V \subset L^\infty\group{\hat{G}}$,
  there exists a $K > 0$
  such that $\norm{u}_{L^\infty\group{\hat{G}}} < K$ holds for all $u \in V$.
  As $\groups{-K, K}$ is compact in $\Rbbb$,
  there exist two constants $\varsigma, \digamma \geq 1$ such that
  \begin{equation*}
    \norms[\big]{\Nfrak^{(n)}(\zeta)}
    \leq (n!)^s \varsigma \digamma^n
  \end{equation*}
  holds for all $n \in \Nbbb$ and all $\zeta \in \groups{-K, K}$,
  by $s$-Gevrey smoothness of $\Nfrak$.

  Now, consider any $u \in V$ and $h \in L^\infty\group{\hat{G}}$
  with $u+h \in V$ and $h \neq 0$. 
  Thus, we can calculate
  \begin{multline*}
    \norm[\big]{N_n(u+h) - N_n(u) - N_{n+1}(u) h}_{L^\infty\group{\hat{G}}} \\
    = \sup_{\xbfm \in \hat{G}} \norms[\Big]{\Nfrak^{(n)}\group[\big]{u(\xbfm)+h(\xbfm)} - \Nfrak^{(n)}\group[\big]{u(\xbfm)} - \Nfrak^{(n+1)}\group[\big]{u(\xbfm)} h(\xbfm)}
  \end{multline*}
  Applying Taylor's formula for $\Nfrak^{(n)}$ yields
  \begin{equation*}
    \Nfrak^{(n)}\group[\big]{u(\xbfm)+h(\xbfm)}
    = \Nfrak^{(n)}\group[\big]{u(\xbfm)} + \Nfrak^{(n+1)}\group[\big]{u(\xbfm)} h(\xbfm) + \frac{1}{2!} \Nfrak^{(n+2)}\group[\big]{\xi_\xbfm} \group[\big]{h(\xbfm)}^2 ,
  \end{equation*}
  where $\xi_\xbfm$ lies in the convex hull of $u(\xbfm)$ and $u(\xbfm)+h(\xbfm)$.
  However, we have that $\xi_\xbfm \in \groups{-K, K}$ holds, and therefore also
  \begin{equation*}
    \norms[\bigg]{\frac{1}{2!} \Nfrak^{(n+2)}\group[\big]{\xi_\xbfm}}
    \leq \frac{\group[\big]{(n+2)!}^s}{2!} \varsigma \digamma^{n+2} .
  \end{equation*}
  This proves that we have
  \begin{equation*}
    \frac{\norm[\big]{N_n(u+h) - N_n(u) - N_{n+1}(u) h}_{L^\infty\group{\hat{G}}}}{\norm{h}_{L^\infty\group{\hat{G}}}}
    \leq \frac{\group[\big]{(n+2)!}^s}{2!} \varsigma \digamma^{n+2} \norm{h}_{L^\infty\group{\hat{G}}}.
  \end{equation*}
  Hence, this implies that $N_n \colon L^\infty\group{\hat{G}} \to L^\infty\group{\hat{G}}$
  is Fr\'echet differentiable with its derivative given by
  \begin{equation*}
    \Dif N_n(u)[h] = N_{n+1}(u) h .
  \end{equation*}
  Noting that $N_0 = N$, we thus inductively have that $N$ is infinitely Fr\'echet
  differentiable for any $u \in V$ and its derivatives are given by
  \begin{equation*}
    \Dif^n N(u) [u_1, \ldots, u_n] = N_n(u) u_1 \cdots u_n .
  \end{equation*}
  As $V$ can be chosen as an open bounded ball around any $u \in L^\infty\group{\hat{G}}$,
  this indeed shows that $N$ is infinitely Fr\'echet differentiable everywhere.

  Finally, noting that
  \begin{equation*}
    \norm[\big]{N_n(u)}_{L^\infty\group{\hat{G}}}
    \leq (n!)^s \varsigma \digamma^n
  \end{equation*}
  holds for all $n \in \Nbbb$ and all $u \in V$
  implies the final assertion, which in turn proves 
  the $s$-Gevrey smoothness of $N$.
\end{proof}

Using the regularity of the Nemyckii operator
$N \colon L^\infty\group{\hat{G}} \to L^\infty\group{\hat{G}}$
in Lemma~\ref{lemma:DN}
and the continuous embeddings $\Ucal \subset L^\infty\group{\hat{G}}$
and $L^\infty\group{\hat{G}} \subset \Rcal$
valid for the choices \eqref{eq:RegSpc} and $m\leq 3$,
implies the following regularity of the residual operator $R$.

\begin{proposition}\label{proposition:DRH2}
  Let $\Nfrak$ satisfy the polynomial growth bound \eqref{eq:Ngrow2}
  and the derivative bounds \eqref{eq:sGN}.

  Then, $R \colon D \times \Ucal \to \Rcal$ is $s$-Gevrey smooth between
  the Banach spaces $\Dcal \times \Ucal$ and $\Rcal$ 
  as in \eqref{eq:RegSpc}.
  Indeed, for any bounded nonempty subsets $B \subset \Dcal$ and $V \subset \Ucal$,
  there exist constants $\varsigma, \digamma \geq 1$ such that
  \begin{equation*}
    \norm[\big]{\Dif^n R(d, u)}_{\Bcal^n\group{\Dcal \times \Ucal; \Rcal}}
    \leq (n!)^s \varsigma \digamma^n
  \end{equation*}
  holds for all $n \in \Nbbb$ and all $d \in B$ and $u \in V$.
\end{proposition}
\begin{proof}
  We recall that, per equation \eqref{eq:residualH2}, $R$ is given by
  \begin{equation*}
    \group[\big]{R(d, u)}(v)
    \isdef \groupa[\big]{\Abfm \Grad u, \Grad v}_{\hat{G}}
    + \groupa[\big]{b N(u), v}_{\hat{G}}
    - \groupa[\big]{f, v}_{\hat{G}} .
  \end{equation*}
  Obviously, the first term in $R(d, u)$ amounts to a bounded bilinear form
  \begin{equation*}
    \Afrak : \Dcal \times \Ucal \to \Rcal, (d, u) \mapsto \group[\Big]{v \mapsto \groupa[\big]{\Abfm \Grad u, \Grad v}_{\hat{G}}}
  \end{equation*}
  and, therefore, is an analytic mapping and thus $s$-Gevrey smooth for every $s\geq 1$.

  Similarily, the third term in $R(d, u)$ amounts to a bounded linear form
  \begin{equation*}
    \Ffrak : \Dcal \times \Ucal \to \Rcal, (d, u) \mapsto \group[\Big]{v \mapsto \groupa[\big]{f, v}_{\hat{G}}} ,
  \end{equation*}
  which also is an analytic mapping and thus $s$-Gevrey smooth  for every $s\geq 1$.

  As the Nemyckii operator
  $N \colon L^\infty\group{\hat{G}} \to L^\infty\group{\hat{G}}$ 
  is $s$-Gevrey smooth per Lemma~\ref{lemma:DN}, by using the 
  Leibniz formula, the mapping
  \begin{equation*}
    \times N\circ : L^\infty\group{\hat{G}} \times L^\infty\group{\hat{G}} \to L^\infty\group{\hat{G}}, (b, u) \mapsto b N(u)
  \end{equation*}
  with $\times$ denoting the bilinear map of pointwise a.e. multiplication
  of pairs of elements in 
  $L^\infty(\hat{G})$ (which is continuous, cf.\ e.g. \cite[Prop.~1.1]{BCD11})
  also is $s$-Gevrey smooth.

  With the continuity of the linear embedding maps
  \begin{equation*}
    \iota_\Ucal : \Ucal \subset L^\infty\group{\hat{G}} \;\;\mbox{and}\;\;
    \iota^*_\Rcal : L^\infty\group{\hat{G}} \subset \Rcal,
  \end{equation*}
  it follows that the composite mapping 
  $\Mfrak := \iota^*_\Rcal \circ (\times N \circ) \circ \iota_\Ucal$
  \begin{equation*}
   \Mfrak :  \Dcal \times \Ucal \to \Rcal, (d, u) \mapsto \group[\Big]{v \mapsto \groupa[\big]{b N(u), v}_{\hat{G}}}
  \end{equation*}
  is $s$-Gevrey smooth.

  By linearity of differentials,
  the sum $\Afrak + \Mfrak + \Ffrak : \Dcal \times \Ucal \to \Rcal$
  is $s$-Gevrey smooth.

  The fact that one can find constants $\varsigma, \digamma \geq 1$
  for any bounded nonempty subsets $B \subset \Dcal$ and $V \subset \Ucal$
  follows by simple bookkeeping of the constants using Lemma~\ref{lemma:DN}.
  This completes the proof.
\end{proof}

Moreover, concerning the inverse of
$\Dif_2 R (d, u) \in \Bcal(\Ucal; \Rcal)$,
we have the following result.
\begin{proposition}\label{proposition:SH2}
  Let $\Nfrak$ satisfy the polynomial growth bound \eqref{eq:Ngrow2}
  and the monotonicity \eqref{eq:Nmonot}.
  Then, $\Dif_2 R (d, u) \in \Bcal(\Ucal; \Rcal)$ is a Banach space isomorphism
  for any $d \in D$ and $u \in \Ucal$.
  Indeed, for any bounded nonempty subsets $B \subset D$ and $V \subset \Ucal$,
  there exists a constant $\alpha \geq 1$ such that
  \begin{equation*}
    \norm[\Big]{\group[\big]{\Dif_2 R\group{d, u}}^{-1}}_{\Bcal\group{\Rcal; \Ucal}}
    \leq \alpha
  \end{equation*}
  holds for all $d \in B$ and $u \in V$.
\end{proposition}
\begin{proof}
  For any $d \in D$ and $u \in \Ucal$, we have that
  $\Dif_2 R (d, u) \in \Bcal(\Ucal; \Rcal)$ is given by
  \begin{equation*}
    \Dif_2 R(d, u)\groups{u_1}
    = - \Div\group{\Abfm \Grad u_1} + b \Dif N(u)[u_1] .
  \end{equation*}
  Thus, we consider the affine residual equation defined by the residual
  $T \colon D \times \Ucal \times L^\infty\group{\hat{G}} \to \Rcal$ given by
  \begin{equation*}
    T\group[\big]{(d,u,w), u_1}
    \isdef - \Div\group{\Abfm \Grad u_1} + b \Dif N(u)[u_1] - w .
  \end{equation*}
  Now, completely analogous arguments as used in the proof of
  Theorem~\ref{theorem:solvH2} prove the stated assertions.
\end{proof}

With these bounds at hand, 
we arrive at the bounds for the Fr\'echet derivatives
of the data-to-solution mapping $S$.
\begin{theorem}\label{theorem:mainH2}
  Let $\Nfrak$ satisfy the polynomial growth bound \eqref{eq:Ngrow2}
  and the monotonicity \eqref{eq:Nmonot}.
  Then, for any bounded nonempty subset $B \subset D$, 
  there exist constants $\varsigma \geq 1$, $\digamma \geq 1$
  and $\alpha \geq 1$ such that, for any $d \in B$, we have
  \begin{equation*}
    \norm[\big]{\Dif^n S(d)}_{\Bcal^n\group{\Dcal; \Ucal}}
    \leq \; (n!)^s \tilde{\varsigma} \tilde{\digamma}^n
  \end{equation*}
  for all \(n \in \Nbbb^*\), 
  where
  \begin{equation*}
    \tilde{\varsigma} = \frac{1}{c_\kappa \alpha \varsigma}
    \quad\text{and}\quad
    \tilde{\digamma} = c_\kappa \alpha^{2} \varsigma^{2} \digamma^{3} \text{.}
  \end{equation*}
\end{theorem}
\begin{proof}
  Simply apply Theorem~\ref{theorem:analyticDS},
  after noting that the Propositions~\ref{proposition:DRH2} and~\ref{proposition:SH2}
  together with equation~\eqref{eq:SboundH2} and Remark~\ref{remark:locality}
  guarantee that its premises are fulfilled.
\end{proof}
%
\subsection{Regularity of the parameters-to-data and the parameters-to-solution mapping.}
Having shown the regularity of the data-to-solution mapping $S$,
we next consider the smoothness of the parameters-to-data mapping 
$\tilde{d}$ and of the parameters-to-solution mapping $\hat{u}$.
To this end, we make the following assumption:
\begin{assumption}\label{assumption:Vbounds}
  The parametric domain mapping $\Vbfm \colon \square \to C^{1,1}(\hat{G}; \Rbbb^m)$
  satisfies \eqref{eq:VBiHoel} and 
  admits bounded mixed partial derivatives of arbitrary order in the sense that,
  for some constants $\mu_{\Vbfm}, \kappa_{\Vbfm} > 0$, $s \geq 1$
  and a $\gammabfm \in \ell^1(\Nbbb^*)$, these satisfy
  \begin{equation*}
    \norm[\big]{\pdif^{\alphabfm} \Vbfm[\ybfm]}_{C^{1,1}(\hat{G}; \Rbbb^m)}
    \leq \group[\big]{\norms{\alphabfm}!}^s \mu_{\Vbfm} \kappa_{\Vbfm}^{\norms{\alphabfm}} \gammabfm^{\alphabfm}
  \end{equation*}
  for all $\ybfm\in \square$ and all multi-indices $\alphabfm$.
\end{assumption}

Now, combining Theorem~\ref{theorem:Gevreyparametriccomp}
together with the results found in \cite[Lemmas~3 and~4]{HS22} 
immediately imply smoothness of the parameters-to-data mapping $\tilde{d}$.
Moreover, as the parametric domain mapping $\Vbfm$ is a bounded 
map by considering $\alphabfm = \zerobfm$, this is also true
for the parameters-to-data mapping $\tilde{d}$.
Therefore, we also have smoothness of the parameters-to-solution
mapping $\hat{u}$ by combining Theorems~\ref{theorem:Gevreyparametriccomp}
and~\ref{theorem:main}, as $\tilde{d}(\square) \subset D$ holds.
Specifically, this gives the following result:
\begin{theorem}\label{theorem:duboundsH2}
  Let $\Nfrak$ be a polynomial nonlinearity satisfying \eqref{eq:Npoly},
  and the monotonicity \eqref{eq:Nmonot}
  and suppose that Assumption~\ref{assumption:Vbounds} holds.
  Then, both, the parameters-to-data mapping $\tilde{d} \colon \square \to D$
  and the parameters-to-solution mapping
  $\hat{u} \colon \square \to H^1_{\mathrm{D}}\group{\hat{G}}$,
  have bounded mixed partial derivatives of arbitrary order
  and there exist constants $\mu_{\tilde{d}}, \kappa_{\tilde{d}} \geq 1$
  and $\mu_{\hat{u}}, \kappa_{\hat{u}} \geq 1$ such that
  \begin{equation*}
    \norm[\big]{\pdif^{\alphabfm} \tilde{d}[\ybfm]}_{\Dcal}
    \leq \group[\big]{\norms{\alphabfm}!}^s \mu_{\tilde{d}} \kappa_{\tilde{d}}^{\norms{\alphabfm}} \gammabfm^{\alphabfm}
  \end{equation*}
  and
  \begin{equation*}
    \norm[\big]{\pdif^{\alphabfm} \hat{u}[\ybfm]}_{H^2(\hat{G})}
    \leq \group[\big]{\norms{\alphabfm}!}^s \mu_{\hat{u}} \kappa_{\hat{u}}^{\norms{\alphabfm}} \gammabfm^{\alphabfm}
  \end{equation*}
  hold for all $\ybfm\in \square$ and all multi-indices $\alphabfm$.
\end{theorem}

\subsection{Nonsmooth reference domain}
\label{sec:RegSolAn}
%
The regularity shift \eqref{eq:regb} is under the 
assumption of homogeneous Dirichlet boundary conditions on $\partial\hat{G}$
that itself is assumed to be $C^{1,1}$-smooth.
For mixed boundary conditions and/or polytopal domain, the
$H^2$ regularity shift in $\hat{G}$ in \eqref{eq:regb}
is known to fail in general. 
A regularity shift only holds in larger, 
corner-weighted (in dimension $m=2$)
or in corner-edge weighted (in dimension $m=3$) 
$H^2(\hat{G})$ spaces of Kondrat'ev type. 
With this choice of spaces in the abstract setting \eqref{eq:RegSpc},
the abstract theory from Section~\ref{sct:abstract} will also apply.

To demonstrate this, we choose to assume that
\begin{equation*}
  m=2 \quad \text{and} \quad \mbox{$\hat{G}$ is a polygon with finite set $\Ccal$ 
    of corner points $\boldsymbol c$.}
\end{equation*}
Then, we require the hilbertian Kondrat'ev spaces $\Kcal^k_a(\hat{G})$ 
given for $k\in \Nbbb_0$ and $a\in \Rbbb$ by 
\begin{equation*}
  \Kcal^k_a(\hat{G}) 
  \isdef \{ u:\hat{G}\to \Rbbb \mid \rho_\Ccal^{|\alphabfm|-a} \partial^\alphabfm u \in L^2(\hat{G}), \; |\alphabfm|\leq k \} . 
\end{equation*}
To also specify the data regularity in $\hat{G}$, we introduce 
\begin{equation*}
  \Wcal^{k,\infty}(\hat{G}) 
  \isdef
  \{
  u:\hat{G}\to \Rbbb \mid \rho_\Ccal^{|\alphabfm|} \partial^\alphabfm u \in L^\infty(\hat{G}), \; |\alphabfm|\leq k
  \} .
\end{equation*}
Here, $\rho_\Ccal(x)>0$ in $\hat{G}$ denotes the product of the distance of $x\in \hat{G}$
to the corners:
\begin{equation*}
  \rho_\Ccal(x) \isdef \prod_{\boldsymbol{c}\in \Ccal} |x-\boldsymbol{c}|, \;\; x\in \hat{G}.
\end{equation*}
Evidently, 
$W^{k,\infty}(\hat{G}) \subseteq \Wcal^{k,\infty}(\hat{G})$ holds for all $k\in \Nbbb_0$.

We now recall the following result from \cite[Theorem~1.1]{BNL17}.
\begin{proposition}\label{prop:K2aReg}
  Assume $m=2$ and that $\hat{G}$ is a polygon 
  with a finite number of straight sides. 
  Assume further in \eqref{eq:PDE} $\hat{\Gamma}_N = \emptyset$, i.e.\
  homogeneous Dirichlet boundary conditions on $\hat{\Gamma}_D$, and 
  that\footnote{Actually, $\rho_\Ccal^2 b \in \Wcal^{1,\infty}\group{\hat{G}; \Rbbb}$ is sufficient.}
  $b \in \Wcal^{1,\infty}\group{\hat{G}; \Rbbb}$.

  Then, the differential operator 
  $P \colon w\mapsto \Div\group{\Abfm \Grad w} + b w$
  is an isomorphism 
  $P\in \Lcal_{is}(\Kcal^2_{1+a}\group{\hat{G}}, \Kcal^0_{a-1}\group{\hat{G}})$
  for $a\in (-a_0,a_0)$ for some $a_0>0$ 
  (depending on $\Abfm$ and on the corner-angles of $\hat{G}$).
  Moreover, the inverse $P^{-1} \in \Lcal_{is}(\Kcal^0_{a-1}\group{\hat{G}},
  \ \Kcal^2_{1+a}\group{\hat{G}})$ depends analytically on the data 
  \begin{equation*}
    d = (\Abfm,b,f) \in 
    \Dcal_r =
    \group[\Big]{\Wcal^{1,\infty}\group{\hat{G}; \Rbbb^{m \times m}_{\textrm{sym}}}}
    \times \Wcal^{1,\infty}\group{\hat{G}; \Rbbb}
    \times \Kcal^0_{a-1}\group{\hat{G}}.
  \end{equation*}
\end{proposition}

Next, we recall that $\Kcal^2_{1+a}(\hat{G})$ is (at least) continuously embedded
into $L^\infty(\hat{G})$, when $a \geq 0$ and $m = 2$, see \cite[Theorem 27 (i)]{DHSS19}.
Therefore, we set the regularity spaces as
\begin{align}\label{eq:RegSpcw}
  \nonumber
  \Dcal_a &\isdef \Wcal^{1,\infty}\group{\hat{G}; \Rbbb^{m \times m}_{\textrm{sym}}}
  \times \Wcal^{1,\infty}\group{\hat{G};\Rbbb}
  \times \Kcal^0_{a-1}\group{\hat{G}},
  \\
  \Ucal_a &\isdef H_0^1\group{\hat{G}} \cap \Kcal^2_{1+a}\group{\hat{G}},
  \\
  \nonumber
  \Rcal_a &\isdef \Kcal^0_{a-1}\group{\hat{G}}
\end{align}
and the subset of admissible data
\begin{equation*}
  D_a \isdef D_{\Abfm}
  \times D_b
  \times \Kcal^0_{a-1}\group{\hat{G}} ,
\end{equation*}
where
\begin{equation*}
  D_{\Abfm} \isdef \groupb[\bigg]{\Abfm \in \Wcal^{1,\infty}\group{\hat{G}; \Rbbb^{m \times m}_{\textrm{sym}}} : \essinf_{\xbfm \in \hat{G}} \min_{\vbfm \in \Rbbb^m \setminus \{\zerobfm\}}\frac{\vbfm^\trans \Abfm(\xbfm) \vbfm}{\vbfm^\trans \vbfm} \geq c_a c_{\Vbfm}^{-m-2} \defis c_\Abfm}
\end{equation*}
and $D_b \isdef \groupb[\big]{b \in \Wcal^{1,\infty}\group{\hat{G};\Rbbb} : \essinf_{\xbfm \in \hat{G}} b(\xbfm) \geq 0}$.

With these definitions and results at hand, one can now obtain
analogous results as in Subsections~\ref{sec:ResEqnH2} and~\ref{sec:RegDatSolH2}.
Therefore, one again is in the abstract setting of Section~\ref{sct:abstract}
and obtains the following result.

\begin{theorem}\label{theorem:mainK2}
  Assume $m=2$, $\hat{G}\subset \Rbbb^2$ is a polygon and
  $a \in [0, a_0)$ with the constant $a_0$ as in Proposition~\ref{prop:K2aReg}.
  Furthermore, let $\Nfrak$ satisfy the polynomial growth bound \eqref{eq:Ngrow2},
  the monotonicity \eqref{eq:Nmonot} and the derivative bounds \eqref{eq:sGN}.

  Then, there is a unique global data-to-solution 
  mapping $S \colon D_a \to \Ucal_a$ such that the equation
  \begin{equation*}
    R\group[\big]{d, S(d)} = 0 \ \text{in}\ \Rcal_a
  \end{equation*}
  is fulfilled for all $d \in D_a$
  and for any bounded nonempty subset $B_a \subset D_a$
  there exist constants $\tilde{\varsigma}, \tilde{\digamma} \geq 1$
  such that
  \begin{equation*}
    \forall d\in B_r \; \forall n\in \Nbbb: \;\;
    \norm[\big]{\Dif^n S(d)}_{\Bcal^n\group{\Dcal_a; \Ucal_a}}
    \leq 
     (n!)^s \tilde{\varsigma} \tilde{\digamma}^n
  \end{equation*}
  holds for all \(n \in \Nbbb\) and $d \in B_a$.
\end{theorem}

\begin{remark}\label{rmk:MixBCs}
  Theorem~\ref{theorem:mainK2} was formulated for homogeneous 
  Dirichlet boundary conditions on all of $\partial \hat{G}$. 
  For mixed boundary conditions as formulated in 
  \eqref{eq:PDE} with $|\hat{\Gamma}_D|>0$ corresponding results hold, 
  referring to \cite[Theorem~5.4]{BNL17}.
\end{remark}
%
\section{Conclusion}\label{sec:conclusio}
In the present article, we investigated the regularity of
mappings between Banach spaces. Our main result here extends
the implicit mapping theorem known for finite smoothness and
holomorphy to the real analytic and Gevrey class situation,
and is not only qualitative but is able to give quantitative bounds
on the Fr\'echet derivatives of the implicit mapping using
quantitative bounds on the Fr\'echet derivatives of the residual mapping,
i.e.\ the mapping which is used to define it. Moreover, we also supplied
results that qualitatively and quantitatively cover the regularity
of the composition of mappings between Banach spaces and the composition
of a mapping between Banach spaces and a (possibly nonlinear) countable
parametric expansion with values in a Banach space.

Applying the quantitative version of the real analytic or
the Gevrey class implicit mapping theorem to residual equations
for partial differential equations amounts to a new methodology
to prove regularity of the dependence of PDE solutions on their data.
In particular, combining this with the quantitative regularity shown
for the composition with a parametric expansion yields the type of
parametric regularity results for PDE solutions that depend on inputs,
which are represented in an affine-parametric manner
in terms of some frame in the data space,
as for example is common in uncertainty quantification.
We illustrated the proposed approach for the specific example of a 
semilinear elliptic PDE defined in a random domain. 
Naturally, the scope of the present approach for uncertainty quantification
is wider: it applies also to other possibly nonlinear PDEs like the 
$p$-Laplacian, Navier-Stokes equations \cite{CSZ18} 
or related eigenvalue problems \cite{CL23}.
The present regularity bounds can constitute the basis for 
a numerical analysis of various discretisation strategies such as 
sparse-grid collocation, polynomial chaos approximation, and Smolyak 
and Quasi-Monte Carlo quadrature in the parametric domain 
of the parametric problem. 
Corresponding Sobolev regularity results in corner-weighted spaces
in $\hat{G}$ are available in \cite{HS22_1031}. Details of single-level
and multi-level algorithms for the quantification of uncertainty in 
nonlinear operator equations will be developed elsewhere.

Furthermore, the (quantitative versions of the) real analytic and
the Gevrey class implicit mapping theorems as well as the theorems
covering the composition of Gevrey class mappings are likely to be useful
also beyond the field of uncertainty quantification.

Lastly, we note that further investigation into the quantitative
bounds of the real analytic and Gevrey class implicit mapping theorems
is of future interest. Especially, the question here is, if the falling factorial
technique used in \cite{CL23} or some other technique enables one
to prove sharper bounds for the derivative, 
or if the bounds we have achieved here are optimal.

\bibliographystyle{plain}

\begin{thebibliography}{10}

\bibitem{AO53}
A.~Alexiewicz and W.~Orlicz.
\newblock Analytic operations in real {B}anach spaces.
\newblock {\em Studia Math.}, 14:57--78, 1953.

\bibitem{AS12}
R.~Andreev and C.~Schwab.
\newblock Sparse tensor approximation of parametric eigenvalue problems.
\newblock In I.G. Graham, T.Y. Hou, O.~Lakkis, and R.~Scheichl, editors, {\em
  Numerical Analysis of Multiscale Problems}, pages 203--241. Springer,
  Berlin-Heidelberg, 2012.

\bibitem{AZ90}
J.~Appell and P.P. Zabrejko.
\newblock {\em Nonlinear Superposition Operators}.
\newblock Cambridge Tracts in Mathematics. Cambridge University Press,
  Cambridge, 1990.

\bibitem{BCD11}
H.~Bahouri, J.-Y. Chemin, and R.~Danchin.
\newblock {\em Fourier analysis and nonlinear partial differential equations},
  volume 343 of {\em Grundlehren der mathematischen Wissenschaften [Fundamental
  Principles of Mathematical Sciences]}.
\newblock Springer, Heidelberg, 2011.

\bibitem{BNL17}
C.~B\u{a}cu\c{t}\u{a}, H.~Li, and V.~Nistor.
\newblock Differential operators on domains with conical points: precise
  uniform regularity estimates.
\newblock {\em Rev. Roumaine Math. Pures Appl.}, 62(3):383--411, 2017.

\bibitem{CNT16}
J.E. Castrill\'on-Cand\'as, F.~Nobile, and R.F. Tempone.
\newblock Analytic regularity and collocation approximation for elliptic {PDEs}
  with random domain deformations.
\newblock {\em Comput. Math. Appl.}, 71(6):1173--1197, 2016.

\bibitem{Chae1985}
S.B. Chae.
\newblock {\em Holomorphy and Calculus in Normed Spaces}.
\newblock Marcel Dekker, Inc., New York and Basel, 1985.

\bibitem{CL23}
A.~Chernov and T.~Le.
\newblock Analytic and {G}evrey class regularity for parametric elliptic
  eigenvalue problems, 2023.
\newblock arXiv:2306.07010.

\bibitem{C2S15}
A.~Chkifa, A.~Cohen, and C.~Schwab.
\newblock Breaking the curse of dimensionality in sparse polynomial
  approximation of parametric {PDEs}.
\newblock {\em J. Math. Pures Appl.}, 103:400--428, 2015.

\bibitem{CD15}
A.~Cohen and R.~DeVore.
\newblock Approximation of high-dimensional parametric {PDE}s.
\newblock {\em Acta Numer.}, 24:1--159, 2015.

\bibitem{CDS10}
A.~Cohen, R.~DeVore, and C.~Schwab.
\newblock Convergence rates of best {$N$}-term {G}alerkin approximations for a
  class of elliptic s{PDE}s.
\newblock {\em Found. Comput. Math.}, 10:615--646, 2010.

\bibitem{CSZ18}
A.~Cohen, C.~Schwab, and J.~Zech.
\newblock Shape holomorphy of the stationary {N}avier--{S}tokes equations.
\newblock {\em SIAM J. Math. Anal.}, 50(2):1720--1752, 2018.

\bibitem{Comtet1974}
L.~Comtet.
\newblock {\em Advanced Combinatorics: The Art of Finite and Infinite
  Expansions}.
\newblock Reidel, Dordrecht, 1974.

\bibitem{CS96}
G.M. Constantine and T.H. Savits.
\newblock A multivariate {F}a\`a di {B}runo formula with applications.
\newblock {\em Trans. Amer. Math. Soc.}, 248:503--520, 1996.

\bibitem{DHSS19}
S.~Dahlke, M.~Hansen, C.~Schneider, and W.~Sickel.
\newblock Properties of {K}ondratiev spaces, 2019.

\bibitem{Deimling}
K.~Deimling.
\newblock {\em Nonlinear Functional Analysis}.
\newblock Springer, Berlin, 1985.

\bibitem{DGLS17}
J.~Dick, R.N. Gantner, Q.T. Le~Gia, and C.~Schwab.
\newblock Multilevel higher-order quasi-{M}onte {C}arlo {B}ayesian estimation.
\newblock {\em Math. Models Methods Appl. Sci.}, 27(5):953--995, 2017.

\bibitem{DickEtc}
J.~Dick, F.Y. Kuo, Q.T. Le~Gia, D.~Nuyens, and C.~Schwab.
\newblock Higher order {QMC} {P}etrov–{G}alerkin discretization for affine
  parametric operator equations with random field inputs.
\newblock {\em SIAM J. Numer. Anal.}, 52(6):2676–2702, 2014.

\bibitem{DGS16}
J.~Dick, Q.T. Le~Gia, and C.~Schwab.
\newblock Higher order quasi-{M}onte {C}arlo integration for holomorphic,
  parametric operator equations.
\newblock {\em SIAM/ASA J. Uncertain. Quantif.}, 4(1):48--79, 2016.

\bibitem{Dieudonne1969}
J.~Dieudonn\'e.
\newblock {\em Foundations of Modern Analysis}, volume~1.
\newblock Academic Press, New York, 1969.

\bibitem{DE22}
J.~D\"olz and D.~Ebert.
\newblock On uncertainty quantification of eigenpairs with higher multiplicity,
  2022.
\newblock arXiv:2210.09089.

\bibitem{GilbargTrudinger}
D.~Gilbarg and N.S. Trudinger.
\newblock {\em Elliptic Partial Differential Equations of Second Order}.
\newblock Springer, Berlin-Heidelberg, 3rd edition, 2001.

\bibitem{GGKSS19}
A.D. Gilbert, I.G. Graham, F.Y. Kuo, R.~Scheichl, and I.H. Sloan.
\newblock Analysis of quasi-{M}onte {C}arlo methods for elliptic eigenvalue
  problems with stochastic coefficients.
\newblock {\em Numer. Math.}, 142:863--915, 2019.

\bibitem{GHM20}
M.~Griebel, H.~Harbrecht, and M.D. Multerer.
\newblock Multilevel quadrature for elliptic parametric partial differential
  equations in case of polygonal approximations of curved domains.
\newblock {\em SIAM J. Numer. Anal.}, 58(1):684--705, 2020.

\bibitem{HHPS18}
A.-L. Haji-Ali, H.~Harbrecht, M.D. Peters, and M.~Siebenmorgen.
\newblock Novel results for the anisotropic sparse grid quadrature.
\newblock {\em J. Complexity}, 47:62--85, 2018.

\bibitem{HL19}
H.~Hakula and M.~Laaksonen.
\newblock Asymptotic convergence of spectral inverse iterations for stochastic
  eigenvalue problems.
\newblock {\em Numer. Math}, 142(3):577–609, 2019.

\bibitem{HS13}
M.~Hansen and C.~Schwab.
\newblock Analytic regularity and nonlinear approximation of a class of
  parametric semilinear elliptic {PDE}s.
\newblock {\em Math. Nachr.}, 286(8--9):832--860, 2013.

\bibitem{HKS23}
H.~Harbrecht, V.~Karnaev, and M.~Schmidlin.
\newblock Quantifying domain uncertainty in linear elasticity.
\newblock Technical Report 2023-06, Fachbereich Mathematik, Universit\"at
  Basel, Switzerland, 2023.

\bibitem{harbrecht2012multilevel}
H.~Harbrecht, M.~Peters, and M.~Siebenmorgen.
\newblock On multilevel quadrature for elliptic stochastic partial differential
  equations.
\newblock In J.~Garcke and M.~Griebel, editors, {\em Sparse {G}rids and
  {A}pplications}, volume~88 of {\em Lecture Notes in Computational Science and
  Engineering}, pages 161--179, Berlin-Heidelberg, 2012. Springer.

\bibitem{HPS16}
H.~Harbrecht, M.~Peters, and M.~Siebenmorgen.
\newblock Analysis of the domain mapping method for elliptic diffusion problems
  on random domains.
\newblock {\em Numer. Math.}, 134(4):823--856, 2016.

\bibitem{HS22}
H.~Harbrecht and M.~Schmidlin.
\newblock Multilevel quadrature for elliptic problems on random domains by the
  coupling of {FEM} and {BEM}.
\newblock {\em Stoch. Partial Differ. Equ. Anal. Comput.}, 10:1619--1650, 2022.

\bibitem{HS22_1031}
Y.~He and C.~Schwab.
\newblock Analytic regularity and solution approximation for a semilinear
  elliptic partial differential equation in a polygon.
\newblock Technical Report 2022-43, Seminar for Applied Mathematics, ETH
  Z{\"u}rich, Switzerland, 2022.

\bibitem{HS3}
R.~Hiptmair, L.~Scarabosio, C.~Schillings, and C.~Schwab.
\newblock Large deformation shape uncertainty quantification in acoustic
  scattering.
\newblock {\em Adv. Comput. Math.}, 44:1475--1518, 2018.

\bibitem{KnuthTAOCPI97}
D.E. Knuth.
\newblock {\em The Art of Computer Programming: Fundamental Algorithms, Volume
  1}.
\newblock Addison Wesley Longman, 3rd edition, 1997.

\bibitem{KN16}
F.Y. Kuo and D.~Nuyens.
\newblock Application of quasi-{M}onte {C}arlo methods to elliptic {PDE}s with
  random diffusion coefficients: a survey of analysis and implementation.
\newblock {\em Found. Comput. Math.}, 16:1631--1696, 2016.

\bibitem{KSS12}
F.Y. Kuo, C.~Schwab, and I.H. Sloan.
\newblock Multi-level {Quasi-Monte Carlo Finite Element Methods} for a class of
  elliptic partial differential equations with random coefficients.
\newblock {\em SIAM J. Numer. Anal.}, 50(6):3351--3374, 2012.

\bibitem{Lang1993}
S.~Lang.
\newblock {\em Real and Functional Analysis}.
\newblock Springer, New York, NY, 3rd edition, 1993.

\bibitem{Schroeder70}
E.~Schr\"oder.
\newblock Vier combinatorische {P}robleme.
\newblock {\em Zeit. f. Math. Phys.}, 15:361--376, 1870.

\bibitem{Stanley1997}
R.P. Stanley.
\newblock {H}ipparchus, {P}lutarch, {S}chr{\"o}der, and {H}ough.
\newblock {\em Am. Math. Mon.}, 104:344--350, 1997.

\bibitem{Taylor38}
A.E. Taylor.
\newblock Additions to the theory of polynomials in normed linear spaces.
\newblock {\em Tohoku Math. J.}, 44:302--318, 1938.

\bibitem{TJWG15}
A.L. Teckentrup, P.~Jantsch, C.G. Webster, and M.~Gunzburger.
\newblock A multilevel stochastic collocation method for partial differential
  equations with random input data.
\newblock {\em SIAM/ASA J. Uncertain. Quantif.}, 3(1):1046--1074, 2015.

\bibitem{XT06}
D.~Xiu and D.M. Tartakovsky.
\newblock Numerical methods for differential equations in random domains.
\newblock {\em SIAM J. Sci. Comput.}, 28(3):1167--1185, 2006.

\bibitem{ZS20}
J.~Zech and C.~Schwab.
\newblock Convergence rates of high dimensional {S}molyak quadrature.
\newblock {\em ESAIM: Math. Model. Numer. Anal.}, 54(4):1259–1307, 2020.

\end{thebibliography}

\end{document}